\documentclass[11pt]{article}                                                                                                                           

\usepackage{color}
\usepackage{latexsym}
\usepackage{amssymb,amsmath,amstext}
\usepackage{epsfig}
\usepackage{graphics}
\usepackage{subfigure} 
\usepackage{euscript}
\usepackage{ifthen}
\usepackage{wrapfig}
\usepackage{amsthm}

\newtheorem{theorem}{Theorem}

\newtheorem{corollary}{Corollary}
\newtheorem{lemma}{Lemma}
\newtheorem{definition}{Definition}

\newtheorem{remark}{Remark}

\newtheorem{problem}{Problem}

\newtheorem{question}{Question}
\newcounter{algo_line}

\newcommand{\Vor}{\mathrm{Vor}}

\begin{document}
{
\title{The Laplacian lattice of a  graph under a simplicial distance function}

\author{Madhusudan Manjunath \\ Universit\"at des Saarlandes and Max-Planck-Institut,\\ Saarbr\"ucken, Germany}

\maketitle

\begin{abstract}

We provide a complete description of important geometric invariants of the Laplacian lattice of a multigraph under the distance function induced by a regular simplex, namely  Voronoi Diagram, Delaunay Triangulation, Delaunay Polytope and its combinatorial structure, Shortest Vectors, Covering and Packing Radius. We use this information to obtain the following results:  i. Every multigraph defines a Delaunay triangulation of its Laplacian lattice and this Delaunay triangulation contains complete information of the multigraph up to isomorphism. ii. The number of multigraphs with a given Laplacian lattice is controlled, in particular upper bounded, by the number of different Delaunay triangulations. iii. We obtain formulas for the covering and packing densities of a Laplacian lattice and deduce that in the space of Laplacian lattices of undirected connected multigraphs, the Laplacian lattices of highly connected multigraphs such as Ramanujan multigraphs possess good covering and packing properties.

           
\end{abstract}






 





\section{Introduction}
The work presented in this paper came as an attempt to obtain a thorough understanding of the Laplacian lattice of a multigraph, 
the lattice generated by the rows of the Laplacian matrix of the multigraph under a simplicial distance function. 
This study played a central role in the geometric proof of the Riemann-Roch theorem for graphs obtained in \cite{AmiMan10} where the study of the Laplacian lattice under 
a simplicial distance function was initiated. The study was focussed towards obtaining a geometric proof of the Riemann-Roch theorem for graphs. 
In this paper, we provide a complete description of important geometric invariants of the Laplacian lattice of a multigraph under the distance function induced by a regular simplex. In particular, we obtain a combinatorial interpretation for important geometric invariants of the Laplacian lattice under the simplicial distance function. 
We use this information to obtain the following results for the Laplacian lattice under the simplicial distance function:
\begin{enumerate}

\item Every connected multigraph defines a Delaunay triangulation of its Laplacian lattice and this Delaunay triangulation contains complete information of the multigraph up to isomorphism. 

\item The number of multigraphs with a given Laplacian lattice is controlled, in particular upper bounded, by the number of different Delaunay triangulations.

\item  We obtain formulas for the covering and packing densities of a Laplacian lattice. Furthermore, we use the combinatorial interpretation of the  geometric invariants of the Laplacian lattice to relate the connectivity properties of the multigraph to the covering and packing density of the corresponding Laplacian lattice in the simplicial distance function. In particular, we show that in the space of Laplacian lattices of undirected connected multigraphs, Laplacian lattices of multigraphs with high-connectivity such as Ramanujan graphs possess good packing and covering density. 
\end{enumerate}

Let us provide some motivation for studying firstly the Laplacian lattice and secondly, for studying its properties under the simplicial distance function
 \footnote{Let $S$ be a full dimensional simplex in $\mathbb{R}^{n}$ that contains the origin, that we call the center of $S$, in its interior. For two points $P_1,~P_2 \in \mathbb{R}^{n}$, the simplicial distance function induced by $S$ is the smallest factor by which we need to scale the translated copy of $S$ centered at $P_1$ so that it contains $P_2$. The Euclidean distance can be obtained by replacing the simplex by a sphere centered at the origin. Given a lattice, the simplicial distance function of a point with respect to a lattice is the minimum simplicial distance of the point to any lattice point.}. The answer is rooted back to the pioneering work of Baker and Norine \cite{BaNo07} where analogues of the Riemann-Roch and Abel-Jacobi theory on algebraic curves were developed on graphs and the geometric of proof of the Riemann-Roch theorem obtained in \cite{AmiMan10}. 
A solitary game played on an undirected connected multigraph known as the chip firing game played a key role in these works and is defined as follows: 
Each vertex of the graph is assigned an integer (positive or negative), that can be thought of as ``chips'' and this assignment is called the initial configuration. 
At each move of the game, an arbitrary vertex $v$ is allowed to either lend or borrow one chip along each edge incident with it and we obtain a new configuration. 
We define two configurations $C_1$ and $C_2$ to be equivalent if $C_1$ can be reached from $C_2$ by a sequence of chip firings. 
The Laplacian matrix $Q(G)$ of the graph naturally comes into the picture as follows:

\begin{lemma}
Configurations $C_1$ and $C_2$ are equivalent if and only if $C_1-C_2$ can be expressed as $Q(G)\cdot w$ for some vector $w$ with integer coordinates.
\end{lemma}

Some natural questions on such a game are:

\begin{enumerate}
\item Is a given configuration equivalent to an effective configuration i.e. a configuration where each vertex has a non-negative number of chips?

\item More generally, given a configuration what is the minimum number of chips that must be removed from the system so that the resulting configuration is not equivalent to an effective configuration?
\end{enumerate}

Let us translate the above notions to the geometric language of lattices and distance functions. The set of all configurations that are equivalent to the origin is the set $\{Q(G) \cdot x|~x \in \mathbb{Z}^{n+1}\}$ where $n+1$ is the number of vertices of the multigraph and this set is precisely the Laplacian lattice $L_G$ of the multigraph. Let $H^{+}_q=\{y \in \mathbb{Z}^{n+1}|~y\geq q\}$ where, for points in $\mathbb{Z}^{n+1}$ we say that $x \geq y$ if $x_i \geq y_i$ for all $i$ from $1$ to $n+1$ and say that $x$ dominates $y$. The set of effective configurations is precisely $H^{+}_O$ where $O$ is the origin. Fix a positive real number $k$, consider the set $E_k$ of configurations that are equivalent to an effective configuration and with total number of chips being equal to $k$. In geometric terms, the set $E_k$ is precisely the set of points with integer coordinates contained in the intersection of the hyperplane $H_k$ given by the equation $\sum_{i=1}^{n+1}x_i=k$ and the union of cones $H^{+}_q$ over all points $q$ in the Laplacian lattice. A simple calculation shows that the intersection of $H^{+}_O$ with the hyperplane $H_k$ is a regular simplex with centroid at the origin. As a consequence, the set $E_k$ is precisely the set of points in $H_k$ with integer coordinates and whose projection onto $H_0$ (the hyperplane containing $L_G$) is contained in an arrangement of regular simplices with centroid at points in the Laplacian lattice. Hence, the definition of effectiveness has an underlying distance function, namely the simplicial distance function. This observation motivated us to study the Laplacian lattice under the simplicial distance function.


\subsection{Related Work}

A substantial body of work is devoted to the study of lattices constructed from graphs, the most well studied ones being the lattice of integral cuts and the lattice of integral flows. This line of investigation was pioneered by the work of Bacher, Harpe and Nagnibeda \cite{BacHarNag97} where they provide a combinatorial interpretation of various parameters of the lattice of integral flows and the lattice of integral cuts in the Euclidean distance function, for example they show that the square norm of the shortest vector of the lattice of integral flows is equal to the girth of the graph. In contrast, the study of the Laplacian lattice under simplicial distance function gives rise to some new and interesting phenomenon, that are not seen in the Euclidean case and in some cases provides more refined information about the graph. Here we note such instances:

i. The Laplacian lattice under the simplicial distance function has different Delaunay triangulations and they contain combinatorial information i.e. the number of different Delaunay triangulations upper bounds the number of multigraphs with the given lattice as their Laplacian lattice. 

ii. The Delaunay triangulation under the simplicial distance function provides more refined information on the underlying (connected) multigraph and in fact characterizes the multigraph completely up to isomorphism. On the other hand in the Euclidean case such an analogue is true only for three-connected graphs. A weaker result is known for general connected graphs \cite{CapViv10}.

iii. The absence of edges between vertices shows up as ``degeneracies'' in the Voronoi diagram of the Laplacian lattice with respect to the simplicial distance function. In particular, multigraphs where every pair of vertices are connected by an edge turn out to be the ``nicest'' cases for the Voronoi diagram of the Laplacian lattice. See Section \ref{VorLap_Subsect} for more details.


On the computational side, Sikric, Sch\"urmann and Vallentin \cite{SikSchVal08} provide an algorithm to compute the Voronoi cell of a lattice. Also, see 
the paper \cite{ErdRys94} for a connection between Voronoi diagrams of lattices and certain regular arrangement of hyperplanes called lattice dicings. In each section, we will further provide references to works that are most relevant to the results in that section.


\subsection{Organisation}
Let us briefly discuss the organisation of the rest of the paper. Section \ref{Prem_Sect} describes some preliminaries and builds up notation useful to develop our main results. We describe the Voronoi neighbours (Section \ref{Vor_Sect}), the packing and covering radius (Section \ref{CovPac_Sect}) and the length of the shortest vector of the Laplacian lattice under the simplicial distance function (Section \ref{Short_Sect}). We then study applications of these results: in Section \ref{Del_Sect} we show that the Delaunay polytope of the Laplacian lattice under the simplicial distance function characterizes the graph completely up to isomorphism, in Section \ref{Count_Sect} we treat the problem of counting the number of graphs with a given Laplacian lattice and in Section \ref{CovPacProb_Sect} we study packing and covering problems on Laplacian lattices.

\section{Preliminaries}\label{Prem_Sect}

In this section, we collect some basic definitions and results that will be useful in the rest of the paper.

\subsection{Lattices} \label{Lat_Sect}

A  {\bf lattice} $L$ is a discrete subgroup of the Euclidean vector space $\mathbb{R}^{n}$. More concretely,  a lattice is the Abelian group obtained by taking all the integral combinations of a set of linearly independent vectors $b_1, \dots,b_k$ in $\mathbb{R}^{n}$. More precisely,
 
\begin{equation}L=\{ \sum_{i=1}^{k} \alpha_i b_i|~\alpha_i \in \mathbb{Z}\} \end{equation}

The set $\mathcal{B}=\{ b_1, \dots,b_k \}$ is called a basis of $L$ and the integer $k$ is called the dimension of $L$, denoted by $dim(L)$, is independent of the choice of the basis. We denote the subspace spanned by the elements of $\mathcal{B}$ by $Span(L)$.

Let us now look at the important geometric invariants of a lattice. The volume of the lattice $Vol(L)$, also known as the discriminant or determinant, is defined as $\sqrt{det(\mathcal{B}\mathcal{B}^{t})}$ where the basis $\mathcal{B}$ is represented as a matrix with its elements row-wise and hence, $\mathcal{B}\mathcal{B}^{t}$ is the Gram matrix of the basis elements. Another important invariant is the norm of the shortest vector of a lattice.

\begin{definition}
The shortest vector of $L$ in the Euclidean norm is an element $q \neq (0,\dots,0)$ of $L$ such that $q \cdot q \leq q' \cdot q'$
for all $q' \in L$ and we denote $||q||_2$ as $\nu_E(L)$. 
\end{definition}

Two other important invariants of a lattice are the packing radius and the covering radius.

\begin{definition}
Let $B(q,R)$ be the Euclidean ball centered at the point $q \in \mathbb{R}$ and with radius $R$.
The packing radius Pac$_E(L)$ and the covering radius Cov$_E(L)$ of a lattice $L$ is defined as: 

Pac$_E(L)= \sup \{R|~B(q_1,R) \cap B(q_2,R)=\emptyset~\forall q_1,~q_2 \in L,~q_1 \neq q_2\}$,

Cov$_E(L)= \inf \{R|~\forall p \in Span(L)~\exists q \in L: p \in B(q,R) \}$.

\end{definition}
The shortest vector and the packing radius are related by Pac$_E(L)=\nu_E(L)/2$. The study of relations between geometric invariants of a lattice is a classical topic \cite{Siegel89}. To convey the flavour of the topic, we state a basic result known as Minkowski's First Theorem:

\begin{theorem}(Minkowski \cite{Min55})
A lattice $L$ has a non-zero element whose Euclidean norm is upper bounded by $2Vol^{1/n}(L)/V_E^{1/n}$
where $V_E$ is the volume of the unit ball in $Span(L)$.
\end{theorem}

\subsection{The Laplacian Lattice of a Graph}

For an undirected connected multigraph $G$, the Laplacian matrix $Q(G)$ is defined as $D(G)-A(G)$ where $D(G)$
is the diagonal matrix with the degree of every vertex in its diagonal and $A(G)$ is the vertex-adjacency matrix of the graph. 
We assume the following standard form of the Laplacian matrix:

\begin{equation}\label{gra_form}
Q=
\begin{bmatrix}
  \delta_0      & -b_{01} & -b_{02}  \hdots & -b_{0n} \\
 -b_{10}   &  \delta_1    & -b_{12}  \hdots & -b_{1n} \\
 \vdots    &  \vdots & \ddots\\
 -b_{n0}   &  -b_{n1} & -b_{n2}  \hdots &  \delta_n
\end{bmatrix}
\end{equation}
has the following properties:
\begin{itemize}
\item[$(C_1)$] $b_{ij}$'s are integers, $b_{ij}\geq0$ for all $0\leq i \neq j\leq n$ and $b_{ij}=b_{ji},\:\:\forall i\neq j$.
\item[$(C_2)$] $\delta_i=\sum_{j=1,j\neq i}^{n}b_{ij}=\sum_{j=1,j\neq i}^{n}b_{ji}$ (and is the degree of the $i$-th vertex).
\end{itemize}

\begin{remark}
Note that we consider undirected connected multigraphs without self-loops and hence, there can be multiple edges between vertices. 
In the rest of paper, when we refer to graphs we actually mean multigraphs.
\end{remark}

Though the Laplacian matrix of a graph contains essentially the same information as the adjacency matrix of a graph, it enjoys other nice properties. See, for example, the chapter ``The Laplacian matrix of a Graph'' in the algebraic graph theory book of Godsil and Royle \cite{GodRoy01} for a more complete discussion.

\begin{lemma}The Laplacian matrix $Q(G)$ is a symmetric positive semi-definite matrix.\end{lemma}

Another remarkable property of the Laplacian matrix is described in the Matrix-Tree theorem:

\begin{theorem} {\bf (Kirchoff's Matrix-Tree Theorem)}The absolute of value of any cofactor is equal to the number of spanning trees of the graph.\end{theorem}

A remarkable aspect of the matrix-tree theorem is that it reduces counting the number of spanning trees of 
a graph into a determinant computation and hence provides a polynomial time algorithm for it.

\begin{definition}{\bf (The Laplacian lattice of a graph)}
Given the Laplacian matrix $Q(G)$, the lattice generated by the rows (or equivalently the columns) of $Q(G)$ is called  the Laplacian lattice $L_G$ of the graph. 
\end{definition}

\begin{definition}{\bf (The Root Lattice $A_n$)}
The root lattice \footnote{Root refers here to root systems in the classification theory of simple Lie algebras~\cite{Bourbaki}}  $A_n$ is the lattice of integer points in the hyperplane $H_0=\{(x_0,\dots,x_n)|~ \sum_{i=0}^{n+1}x_i=0,~x_i \in \mathbb{R}\}$.
More precisely, 
\begin{equation}\notag A_n=\{(x_0,\dots,x_n)|~ \sum_{i=0}^{n+1}x_i=0,~x_i \in \mathbb{Z}\}.\end{equation}
\end{definition}


We now make a few simple observations on the Laplacian lattice of a graph.

\begin{lemma} The Laplacian lattice of a graph on $n+1$-vertices is a sublattice of the root lattice $A_n$.\end{lemma}

\begin{definition}{\bf (The covolume of a sublattice)}
A full-dimensional sublattice $L_s$ of a lattice $L$ is a subgroup of the Abelian group $L$ 
the cardinality of the quotient group $L/L_s$ is called the covolume of $L_s$ with respect to $L$. \end{definition}

\begin{lemma}\label{covol_lem}\cite{Sho09} The covolume of the Laplacian lattice of $G$ with respect to $A_n$ is equal to the number of spanning trees of $G$.\end{lemma}

The elements of $A_n/L_G$ naturally possess an Abelian group structure and this group is known as the Picard group of $G$, also known as the Jacobian of $G$. A number of works have been devoted to the study of the structure of this group and 
the information that it contains about the underlying graph, see for example the works of Biggs \cite{Big97}, Kotani and Sunada \cite{KotSun98} and Lorenzini \cite{Lor08}. As a straightforward corollary to Lemma \ref{covol_lem} we obtain:

\begin{corollary}
The cardinality of the Picard group of $G$ is equal to the number of spanning trees of $G$.
\end{corollary}



\subsection{Polyhedral Distance Functions}

 Let $\mathcal{P}$ be a convex polytope in $\mathbb R^n$ with the reference point $O=(0,\dots,0)$ that we call the ``center'' in its interior. By $\mathcal{P}(p,\lambda)$ we denote a dilation of $\mathcal{P}$ by a factor $\lambda$ and its center translated to the point $p$ i.e. $\mathcal{P}(p,\lambda)=p+\lambda.\mathcal{P}$ and $\lambda.\mathcal{P}\:=\:\{\:\lambda.x\:|\:x\in \mathcal{P}\:\}$. 
We define the $\mathcal{P}$-midpoint of two points $p$ and $q$ in $\mathbb{R}^{n}$ as $\inf\{R|~\mathcal{P}(p,R) \cap \mathcal{P}(q,R) \neq \emptyset\}$. 
The {\it polyhedral distance function} $d_{\mathcal{P}}(.\:,.)$ between the points of $\mathbb R^n$ is defined as follows:
\[\forall\: p,q\in \mathbb R^n,\: d_{\mathcal{P}}(p,q)\::=\:\inf\{\lambda\geq 0\:|\:q \in \mathcal{P}(p,\lambda)\}.\]
$d_{\mathcal{P}}$ is not generally symmetric, indeed it is easy to check that $d_{\mathcal{P}}(.\:,.)$ is symmetric if and only if the polyhedron $\mathcal{P}$ is centrally symmetric i.e. $\mathcal{P}=-\mathcal{P}$. Nevertheless $d_{\mathcal{P}}(.\:,.)$ satisfies the triangle inequality. 

\begin{lemma}\label{lin_lem}
For every three points $p,q,r\in \mathbb R^{n}$, we have $d_{\mathcal{P}}(p,q)+d_{\mathcal{P}}(q,r) \geq d_{\mathcal{P}}(p,r)$.
 In addition, if $q$ is a convex combination of $p$ and $r$, then $d_{\mathcal{P}}(p,q)+d_{\mathcal{P}}(q,r)=d_{\mathcal{P}}(p,r)$.
\end{lemma}

We also observe that the polyhedral metric $d_{\mathcal{P}}(.\:,.)$ is translation invariant, i.e.
\begin{lemma}\label{lem:trans-inv} For any two points $p,q$ in $\mathbb R^n$, and for any vector $v\in\mathbb R^n$, we have $d_{\mathcal{P}}(p,q) = d_{\mathcal{P}}(p-v,q-v)$. In particular, $d_{\mathcal{P}}(p,q)=d_{\mathcal{P}}(p-q,O)=d_{\mathcal{P}}(O,q-p)$. 
\end{lemma}

\begin{remark}\rm
The notion of a polyhedral distance function is essentially the concept of a gauge function of a convex body that has been studied in \cite{Siegel89}. Lemmas \ref{lin_lem} and \ref{lem:trans-inv} can be derived in a straightforward way from the results in \cite{Siegel89}. 
\end{remark}

Recall that the $n$-dimensional hyperplane $H_0$ is defined as $H_0=\{(x_0,\dots,x_n)|~ \sum_{i=0}^{n}x_i=0\}$. 
We will be mainly be using distance functions defined by regular simplices $\triangle$ and $\bar{\triangle}$ where the simplices $\triangle$ and $\bar{\triangle}$ are defined as follows: 

\begin{definition}\label{regsimp_def}
The regular simplex $\triangle$ is the convex hull of $t_0,\dots,t_{n} \in H_0$ where 
\label{deltaver_eq}\begin{equation}\notag
t_{ij}=
\begin{cases}
~~n, \text{~if}~i=j,\\
-1, \text{~otherwise}
\end{cases}
\end{equation}
 for $i$ from $1,\dots,n$ and $t_{ij}$ is the $j$-th coordinate of $t_i$. We define $\bar{\triangle}$ as $-\triangle$.
\end{definition}

We note that for points in $H_0$, the distance functions $d_{\triangle}$ and $d_{\bar{\triangle}}$ have a simple formula:

\begin{lemma} (Lemma 4.7, \cite{AmiMan10})
For any pair of points $p,~q$ in $H_0$, we have:
              \begin{gather}
           \notag d_{\triangle}(p,q)=|\min_i (q_i-p_i)|,\\
          \notag d_{\bar{\triangle}}(p,q)=|\min_i (p_i-q_i)|.
         \end{gather}
\end{lemma}

\begin{remark}
The notion of a simplicial distance function with respect to a lattice is sometimes captured in the language of 
``lattice point free simplices'', see \cite{Sca08} for more details on this viewpoint. The author is indebted to Bernd Sturmfels
 for pointing out this connection.
\end{remark}

\subsection{Geometric Invariants of a Lattice with respect to Polyhedral Distance Functions}

In Section \ref{Lat_Sect}, we defined some important geometric invariants of a lattice with respect to the Euclidean norm.
We define analogous invariants with respect to polyhedral distance functions.

\begin{definition}
An element $q$ of $L$ is called a shortest vector with respect to the polyhedral distance function $d_{\mathcal{P}}$ if
$d_{\mathcal{P}}(O,q) \leq d_{\mathcal{P}}(O,q')$ for all $q' \in L/\{O\}$, where $O$ is the origin.  
We denote $d_{\mathcal{P}}(O,q)$ by $\nu_{\mathcal{P}}(L)$.
\end{definition}

Note that the shortest vector with respect the distance functions $d_{\mathcal{P}}$
and $d_{\bar{\mathcal{P}}}$ can be potentially different.

\begin{definition}
For a lattice $L$, we define packing and covering radius of $L$ with respect $\mathcal{P}$  as:

Pac$_{\mathcal{P}}(L)= \sup \{R|~\mathcal{P}(q_1,R) \cap \mathcal{P}(q_2,R)=\emptyset,~ \forall q_1,~q_2 \in L ~q_1 \neq q_2\}$

Cov$_{\mathcal{P}}(L)= \inf \{R|~\forall p \in Span(L) \text{ is contained in } \mathcal{P}(q,R) \text{ for some } q \in L \}$

\end{definition}

As we shall observe, for polyhedral distance functions it is no longer true that Pac$_{\mathcal{P}}(L)=\nu_{\mathcal{P}}(L)/2$.


\subsubsection{Voronoi Diagrams}

\noindent Consider a discrete subset $\mathcal S$ in $\mathbb{R}^{n}$. For a point $s$ in $\mathcal S$, we define the {\it Voronoi cell} of $s$ with respect to the distance function $d_{\mathcal{P}}$  as $V_{\mathcal{P}}(s)\:=\:\{\: p \in \mathbb{R}^{n}\:|~\:d_{\mathcal{P}}(p,s)\leq d_{\mathcal{P}}(p,s')\:\: \textrm{~for any other point~}\:\: s'\in \mathcal S\:\}\:.$ Unlike the Euclidean case, the Voronoi cell of a point is not necessarily a convex set, see Figure 1 for an example, but nevertheless the following weaker property holds.  

\begin{figure}[!htb]\label{voronoi_fig}
\begin{center}
\includegraphics[width=0.25\linewidth]{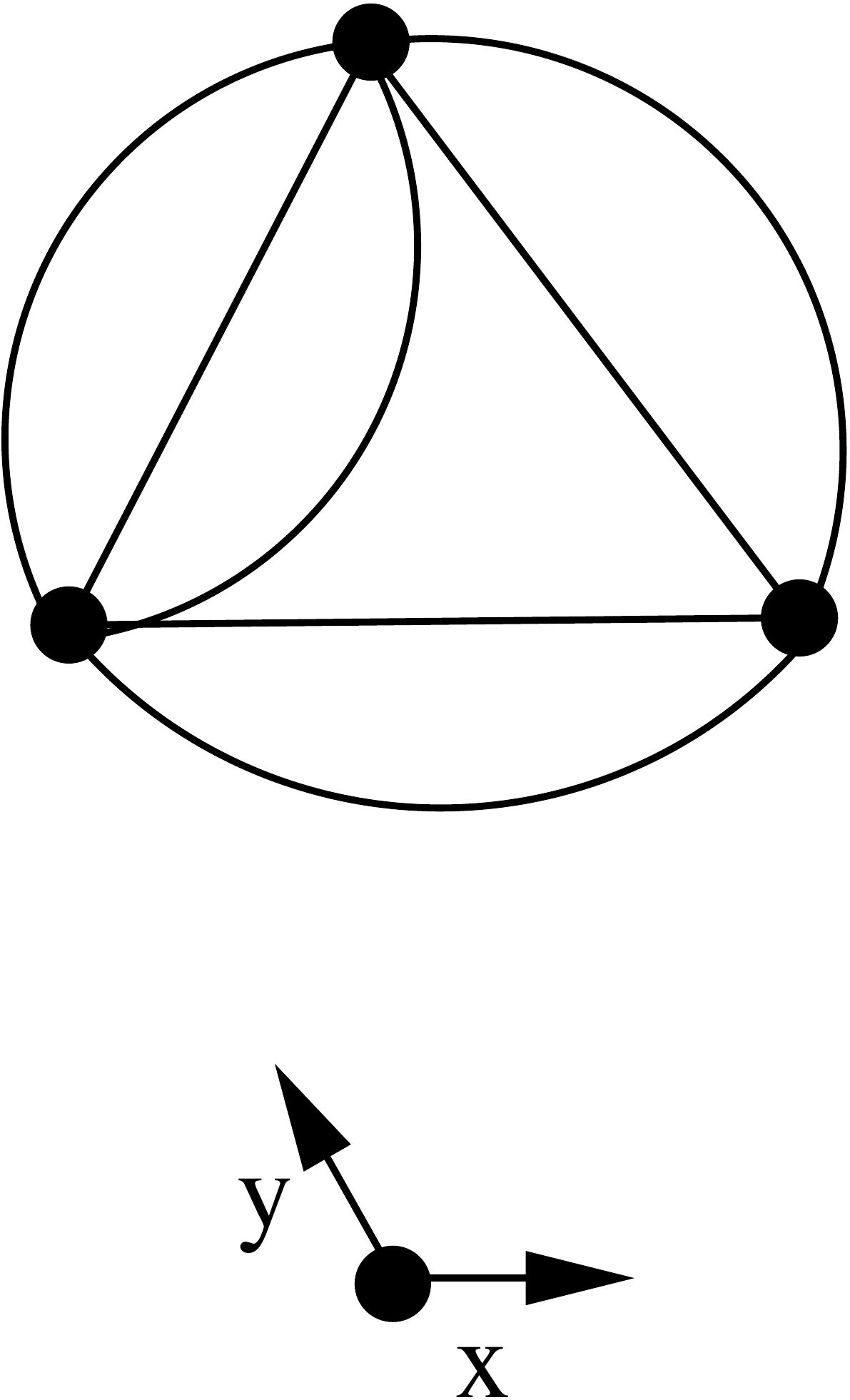}
\hspace{3cm}
\includegraphics[width=0.45\linewidth]{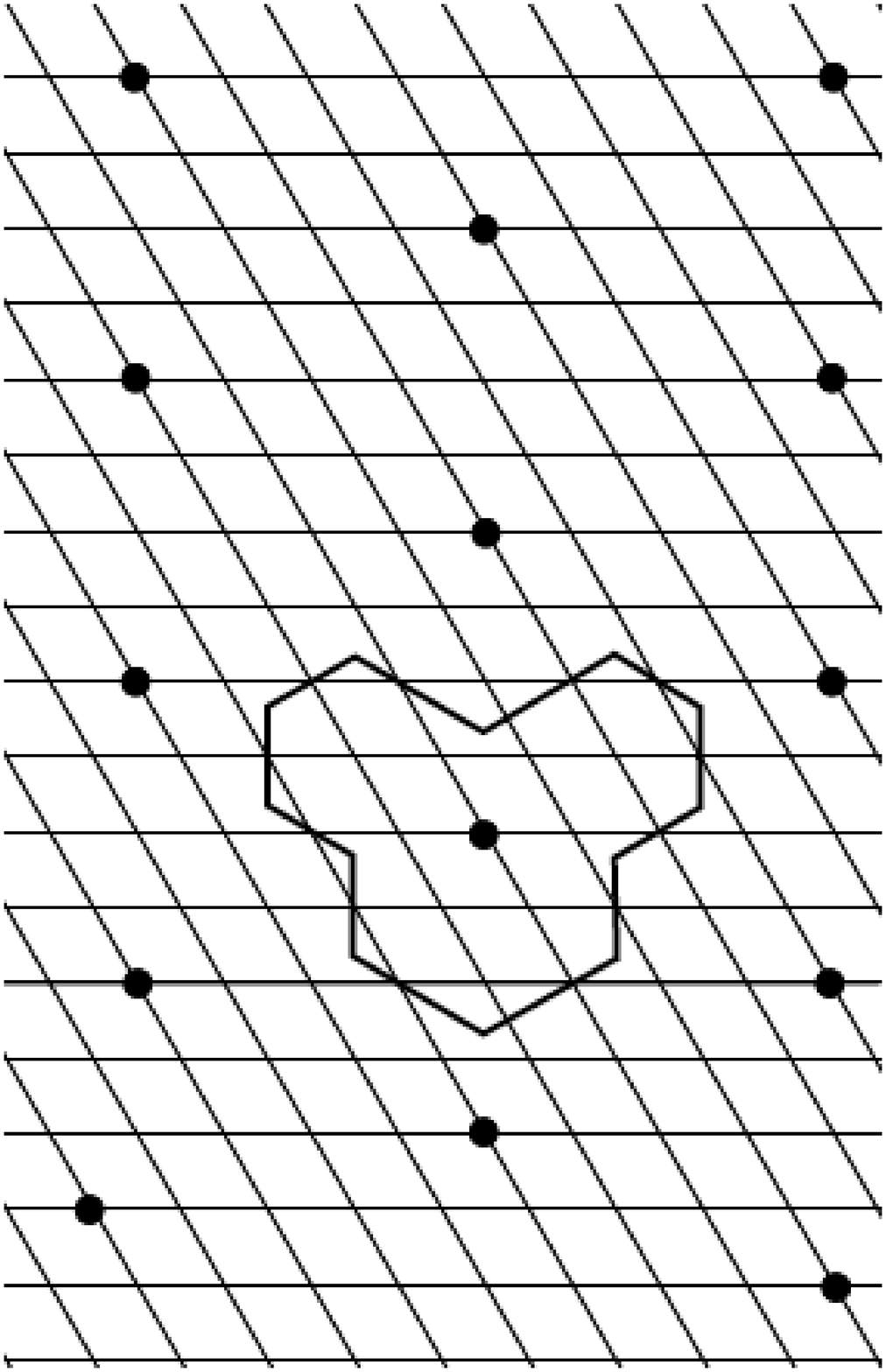}
\caption{The shape of a Voronoi-cell in the Laplacian lattice of a graph with three vertices. The multi-graph $G$ has three vertices and $7$ edges. The lattice $A_2$ is generated by the two vectors $x = (1,-1,0)$ and $y=(-1,0,1)$. The  corresponding Laplacian sublattice of $A_2$, whose elements are denoted by $\bullet$, is generated by the vectors $(-5,3,2)=-3x+2y$ and $(3,-5,2) = 5x+2y$ (and $(2,2,-4) = -2x-4y$), which correspond to the vertices of $G$. The Delaunay triangulation 
consists of all translates of the triangle with vertices $(0,0,0)$, $(-5,3,2)$ and $(-2,-2,4)$ and the triangle with vertices $(0,0,0)$, $(3,-5,2)$ and $(-2,-2,4)$ by points in the Laplacian lattice of $G$.}
\end{center}
\end{figure}


\begin{lemma}\label{lem:star} Let $\mathcal S$ be a discrete subset of $\mathbb R^n$ and $\Vor_{\mathcal{P}}(\mathcal S)$ be the Voronoi cell decomposition of $\mathbb R^n$. For any point $s$ in $\mathcal S$, the Voronoi cell $V_{\mathcal{P}}(s)$ is a star-shaped polyhedron with $s$ as a kernel. \end{lemma}

\begin{lemma} Let $\mathcal S$ be a discrete subset of $\mathbb R^n$. For any point $s$ in $\mathcal  S$ we have: $\Vor_{\mathcal{P}}(s)=-\Vor_{-\mathcal{P}}(s)$.\end{lemma}

\begin{definition} {\bf (Voronoi Neighbours)}
We say that distinct points $p,~q \in S$  are Voronoi neighbours if the intersection of their Voronoi cells are non-empty.
\end{definition}

\begin{lemma} (Lemma 4.6 of \cite{AmiMan10})\label{comp_lem} For the Voronoi cell decomposition of a full-dimensional lattice in $\mathbb{R}^{n}$ with respect to a polyhedral distance function $d_Q$, the Voronoi cell of every lattice point is compact.
\end{lemma}

\begin{lemma}\label{shortvec_lem} Any shortest vector of a lattice $L$ under the polyhedral distance function $d_{\mathcal{P}}$ 
is a Voronoi neighbour of the origin under the same distance function.\end{lemma}

\begin{proof}
Let $p$ be a shortest vector of $L$ in the distance function $d_{\mathcal{P}}$. 
Consider the intersection $p_I$ of the ray $\overrightarrow{Op}$ with the Voronoi cell $V_{\mathcal{P}}(O)$
of the origin under the distance function $d_{\mathcal{P}}$ and by Lemma \ref{comp_lem}, we know that the ray $\overrightarrow{Op}$ 
intersects $V_{\mathcal{P}}(O)$ at some point $p_I \neq p$. By Lemma \ref{lin_lem}, we have:
  $d_{\mathcal{P}}(O,p)=d_{\mathcal{P}}(O,p_I)+d_{\mathcal{P}}(p_I,p)$. Assume for contradiction that $p$
is not a Voronoi neighbour of the origin under the distance function $d_{\mathcal{P}}$.
This means that there is a lattice point $p'$ such that $d_{\mathcal{P}}(p_I,p')<d_{\mathcal{P}}(p_I,p)$.
Hence we have: 
\begin{equation}
d_{\mathcal{P}}(O,p') \leq d_{\mathcal{P}}(O,p_I) + d_{\mathcal{P}}(p_I,p')<d_{\mathcal{P}}(O,p_I) + d_{\mathcal{P}}(p_I,p)=d_{\mathcal{P}}(O,p).
\end{equation}
This contradicts our assumption that $p$ is a shortest vector of $L$.
\end{proof}

\subsubsection{Distance function induced by a Discrete Point Set}

Given a polyhedral distance function $\mathcal{P}$ and a discrete point set $S$, we define a function $h_{\mathcal{P},S}:\mathbb{R}^{n} \rightarrow \mathbb{R}$
as:
\begin{equation} h_{\mathcal{P},S}(p)= \min_{q \in S}\{d_{\mathcal{P}}(p,q)\}\end{equation}

In particular, the notion of local minima and local maxima of the distance function $h_{\mathcal{P},S}$ turns out to be useful:

\begin{definition}{\bf{(Local Maxima and Local Minima of $h_{\mathcal{P},S}$)}}
Let  $B(p,\epsilon)$ be the Euclidean ball of radius $\epsilon$ centered at $p$.
A point $c$ in $\mathbb{R}^{n}$ is called a local minimum of $h_{\mathcal{P},S}$ if there exists an $\epsilon>0$ such that
  $h_{\mathcal{P},S}(c) \leq h_{\mathcal{P},S}(q)$ for all $q \in B(c,\epsilon)$. A point $c$ in $\mathbb{R}^{n}$ is called a local maximum of $h_{\mathcal{P},S}$ if there exists an $\epsilon>0$ such that $h_{\mathcal{P},S}(c) \geq h_{\mathcal{P},S}(q)$ for all $q \in B(c,\epsilon)$.
\end{definition}

We have the following characterization of local minima of $h_{\mathcal{P},S}$:

\begin{lemma} A point $q \in \mathbb{R}^{n}$ is a local minimum of $h_{\mathcal{P},S}$ if and only if $q \in S$.\end{lemma}

We denote the set of local maxima of $h_{\mathcal{P},S}$ by Crit$_{\mathcal{P}}(S)$. The characterization of the local maximum of $h_{\mathcal{P},S}$ turns out to be more complex. In general, the local maxima of $h_{\mathcal{P},S}$ are a subset of Voronoi vertices of $S$ with respect of $-\mathcal{P}$. We will see in the next section that the local maxima of the Laplacian lattice with respect to $\triangle$ are essentially the acyclic orientations of $G$.


\subsubsection{Delaunay Triangulations}\label{Deltri_Subsect}

We define the Delaunay triangulation of a point set under a polyhedral distance function 
as follows:

\begin{definition}{\bf (Delaunay Triangulation under a Polyhedral Distance Function)}\label{Del_def}
A triangulation $\mathcal{T}$ of a discrete point set $S$ in $\mathbb{R}^d$ is a Delaunay 
triangulation of $S$ under the polyhedral distance function $d_{\mathcal{S}}$ if for every 
point $c$ in Crit$_{\mathcal{P}}(S)$ there exists a simplex $K$ in $\mathcal{T}$ such that 
 $Q(c,h_{\mathcal{P},S}(c))$ contains the vertices of $K$ in its boundary.\end{definition}

The Delaunay triangulation under the simplicial distance function $d_{\triangle}$ is closely related to the notion of Scarf complex associated with a lattice \cite{Sca08}. 
In fact, in the case of multigraphs whose Laplacian lattice has no zero entries, the Scarf complex 
coincides with the Delaunay triangulation, but in general the Scarf complex is a subcomplex of the 
Delaunay triangulation.  The Scarf complex of a lattice was first considered in the context of mathematical economics and integer programming, and was later used in 
commutative algebra to determine free resolutions of the associated lattice ideal. 
The author is indebted to Bernd Sturmfels for pointing out this connection to him.

\begin{remark}
Note that Definition \ref{Del_def} is a natural extension of the notion of 
Delaunay triangulation under the Euclidean distance. The Delaunay triangulation 
under the Euclidean distance has the property that every Voronoi vertex is 
the circumcenter of some Delaunay simplex. See a book on discrete geometry, 
for example \cite{BeChKrOv08} for more details. 
\end{remark}

In the following sections, we undertake a detailed study of the Laplacian lattice under the simplicial distance function $d_{\triangle}$. Though we do not always mention it explicitly we always assume that the underlying distance is the simplicial distance function $d_{\triangle}$.

\section{Voronoi Diagram and Delaunay triangulation of the 
        Laplacian lattice under the simplicial distance function} \label{Vor_Sect}



In this section, we describe the Voronoi diagram, Delaunay triangulation and the local maxima of the simplicial distance function induced by the distance function $d_{\triangle}$ on the Laplacian lattice. We draw heavily from Section 6 of \cite{AmiMan10}. This section is intended to serve two purposes, it contributes to providing a complete description of the geometry of the Laplacian lattice under the simplicial distance function $d_{\triangle}$ and secondly, the results in this section are frequently used in the rest of the paper. 

Let $\{b_0,\dots,b_{n}\}$ be the rows of the Laplacian matrix of $G$. For each permutation $\sigma \in S_{n+1}$, where $S_{n+1}$ is the symmetric group on $n+1$-elements. 
we define 

   \begin{equation}u^{\sigma}_i=\sum_{j=0}^{i-1}b_{\sigma(i)}\end{equation} 

for $i$ from $0$ to $n$. Note that $u^{\pi}_n=O$ for every permutation $\pi \in S_{n+1}$.


\subsection{Local Maxima of the simplicial distance function}

Given a permutation $\pi$ on the $n+1$ vertices of $G$, define the ordering $\pi(v_0) <_{\pi} \pi(v_1)<_{\pi} \dots <_{\pi}\pi(v_n)$ and orient the edges of graph $G$ according to the ordering defined by $\pi$ i.e. there is an oriented edge from $v_i$ to $v_j$ if $(v_i,v_j) \in E$ and if $v_i<_{\pi}v_j$ in the ordering defined by $\pi$. Consider the acyclic orientation induced by a permutation $\pi$ on the set of vertices of $G$ and define $\nu_{\pi}=(indeg_{\pi}(v_0),\dots,indeg_{\pi}(v_{n}))$, where $indeg_{\pi}(v)$ is the indegree of the vertex $v$ in the directed graph oriented according to $\pi$. Define Ext$^{c}(L_G)= \{\nu_{\pi}+q|~ \pi \in S_{n+1},~ q \in L_G\}$.

\begin{theorem} (Theorem 6.1 of \cite{AmiMan10}) The elements of the local maxima of the simplicial distance function 
$h_{\triangle,L_G}$ are precisely the orthogonal projections of the 
elements of Ext$^{c}(L_G)$ onto $H_0$.\end{theorem}


\subsection{Voronoi Diagram}\label{VorLap_Subsect}

We shall first consider the case of Laplacian lattices generated by multigraphs 
where every pair of vertices are connected by an edge (for example, the complete graph). As we shall see, these graphs are technically easier to handle than the general case. Recall that distinct lattice points $p,~q$ are Voronoi neighbours if the intersection of their Voronoi cells are non-empty. 

\begin{theorem}\label{NiceVor_theo}{(Corollary 6.18, \cite{AmiMan10})}
For the Laplacian lattice of a graph for which every pair of vertices are connected by an edge, 
a lattice point $q$ is a Voronoi neighbour of the origin with respect to $d_{\triangle}$ if and only if $q$ is of the 
form $u^{\sigma}_i$ for some $\sigma \in S_{n+1}$ and an integer $i$ from $1$ to $n-1$.
\end{theorem}


The case of general multigraphs is slightly more involved and are a consequence of results of \cite{AmiMan10}.
Since Theorem \ref{VorNeigh_theo} is not used in the rest of the paper, a reader may skip the theorem for the first reading.

\begin{theorem}\label{VorNeigh_theo}
The Voronoi neighbours of the origin with respect to $d_{\triangle}$ are precisely 
the set of non-zero lattice points that are contained in $\bar{\triangle}(c_{\pi},h_{\triangle,L_G}(c_{\pi}))$ where
$c_{\pi}$ is the orthogonal projection of $\nu_{\pi}$ onto $H_0$ and $\bar{\triangle}(c_{\pi},h_{\triangle,L_G}(c_{\pi}))$ is the simplex $h_{\triangle,L_G}(c_{\pi}) \cdot \bar{\triangle}+c_{\pi}$ (see Definition \ref{regsimp_def}).
\end{theorem}

\begin{proof} First, we know that the Voronoi cell of every non-zero lattice point $q$ in $\bar{\triangle}(c_{\pi},h_{\triangle,L_G}(c_{\pi}))$ shares $c_{\pi}$ with the Voronoi cell of the origin and hence $q$ is a Voronoi neighbour of the origin. Conversely, consider a lattice point $q$ that is a Voronoi neighbour of the origin $O$ we know that the simplices centered at $q$ and $O$ share a point $m$ say in the boundary of the arrangement of simplices $\triangle$ centered at lattice points and the radius of simplices is $r$ for some real number $r>0$. By the duality theorem (Theorem 8.3 of \cite{AmiMan10}), we know that there is a point $c$ in Crit$_{\triangle}(L_G)$ such that $\bar{\triangle}(c,\mathrm{Cov}_{\triangle}(L_G)-r)$ contains $m$. Applying triangle inequality, we deduce that $d_{\bar{\triangle}}(c,O) \leq \mathrm{Cov}_{\triangle}(L_G)$ but we also know that $d_{\bar{\triangle}}(c,q') \geq \mathrm{Cov}_{\triangle}(L_G)$ for all $q' \in L_G$. Hence, $d_{\bar{\triangle}}(c,O)=d_{\triangle}(O,c)=\mathrm{Cov}_{\triangle}(L_G)$. Using item ii of Theorem 6.9 of \cite{AmiMan10} we know that $c=c_{\pi}$ for some permutation $\pi \in S_{n+1}$. Similarly, we also know that $d_{\bar{\triangle}}(c,q) = \mathrm{Cov}_{\triangle}(L_G)$ and hence, $q$ is contained in $\bar{\triangle}(c_{\pi},h_{\triangle,L_G}(c_{\pi}))$. \end{proof} 





\subsection{A perturbation trick}
For the case of graphs where every pair of vertices are connected by an edge, 
Theorem \ref{NiceVor_theo} gives a useful characterization of the Voronoi neighbours of the origin, but for the case of general connected graphs, the characterization obtained in Theorem \ref{VorNeigh_theo} is not explicit enough for our purposes. In the subsequent sections, we use the following perturbation trick to handle the Laplacian lattice of a general connected graph: we perturb the Laplacian lattice and scale it to the ``nice'' case, i.e. to the case of lattices generated by graphs where every pair of vertices are connected by an edge and study the limit as the perturbation tends to zero. More precisely, we consider lattices generated by the following perturbed basis: we add a rational number $\epsilon>0$ to every non-diagonal element $b_{ij}$ of the Laplacian matrix and then set the diagonal elements so that the row sum and column sum is zero. We call such a perturbation a {\bf standard} perturbation and denote the vector obtained by perturbing $b_i$ by $b^{\epsilon}_i$.  The following lemma characterizes the Voronoi neighbours of the perturbed Laplacian lattice: 

\begin{lemma}\label{pert_lem} Let $L_G$ be the Laplacian lattice of the graph and $L^{\epsilon}_G$ be the lattice obtained by perturbing $L_G$ according to the standard perturbation. The Voronoi neighbours of $L^{\epsilon}$ under the distance function $d_{\triangle}$ are of the form $u^{\epsilon}_S=b^{\epsilon}_{i_1}+b^{\epsilon}_{i_2}+\dots+b^{\epsilon}_{i_k}$ where $i_j$s are distinct.  
\end{lemma}
\begin{proof}
 Since the perturbation $\epsilon$ is rational, we can scale the lattice by a factor $\lambda$, say to obtain the Laplacian lattice of a graph where every pair of vertices is connected by an edge.  We can apply Theorem \ref{NiceVor_theo} to the scaled lattice $\lambda L^{\epsilon}$ to deduce that the Voronoi neighbours of the lattice $\lambda L^{\epsilon}$ are of the form $\lambda u^{\epsilon}_S$ for some subset $S$. Finally, observe that the Voronoi neighbours of a lattice are preserved under scaling to complete the proof.
\end{proof}

\begin{remark}
Note that there are different ways of perturbing the basis to obtain a ``nice'' lattice. For example, we can add an $\epsilon>0$ only to the non-diagonal elements that are zero and then set the diagonal elements such that the row sum and column sum is zero.
\end{remark}

\subsection{Delaunay Triangulation}\label{DeltriLap_Subsect}

\begin{theorem}(See Section 6.1 of \cite{AmiMan10}) \label{LapDel_theo} Let $S^{\sigma}$ be the convex hull of $u^{\sigma}_0,\dots,u^{\sigma}_{n}$. The set $\{S^{\sigma}+p\}_{\sigma \in S_{n+1}, p \in L_G}$ is a Delaunay triangulation of $L_G$ and is a unique Delaunay triangulation if $G$ is a graph that contains the complete graph. \end{theorem}

We shall see in Section \ref{Count_Sect} that the the number of different graphs that 
have $L_G$ as their Laplacian lattice upper bounds the number of ``different'' 
Delaunay triangulations of $L_G$.













\section{Packing and Covering Radius of the Laplacian lattice}\label{CovPac_Sect}

We will show that the packing radius of the graph under the distance function $\triangle$ (and $\bar{\triangle}$)
is essentially (up to a factor depending on the number of vertices) the minimum cut of a graph.

\begin{definition} {\bf ($\ell_{1}$-Minimum Cut)}\label{ell1cut_def}
For a non-trivial cut $S$ of $V(G)$ i.e. $S$ is neither empty nor equal to $V(G)$, define the weight of the cut 
$\mu_1(S)=\sum_{v \in S}deg_{S,\bar{S}}(v)$ where $deg_{S,\bar{S}}(v)$ is the degree of the vertex $v$
across the cut $S$. Now define the $\ell_{1}$-minimum cut MC$_{1}(G)$ as the minimum of $\mu_1(S)$ over all non-trivial cuts $S$ .
\end{definition}

Remark that $\ell_1$-minimum cut of a graph is the same as the minimum cut of a graph. 
We call it the $\ell_1$-minimum cut to distinguish from a variant the ``$\ell_{\infty}$-minimum cut'' that we will encounter in the next section. 

For points $p,q \in \mathbb{R}^{n+1}$, the max-sum $p \oplus q$ is defined as $(max(p_1,q_1),\dots,max(p_{n+1},q_{n+1})$.


\begin{lemma}\label{radtrop_lem}
For any point $p$ in $H_0$, $\triangle$-midpoint $m$ of $p$ and the origin $O$ is the projection of the max-sum of the two points onto $H_0$ and $d_{\triangle}(p,m)=d_{\triangle}(O,m)=||p \oplus O||_1/(n+1)$.
\end{lemma}
\begin{proof}
For a point $r \in \mathbb{R}^{n+1}$ let $H^{+}_r$ be the domination cone defined as: $H^{+}_r=\{r'|~r' \in \mathbb{R}^{n+1},~r'-r \geq O\}$. Consider $H^{+}_p$ and $H^{+}_O$ and observe that the point with minimum $\ell$-norm in their intersection is $p \oplus O$. Project the system onto $H_0$ along the normal $(1,\dots,1)$. A simple computation shows that the projection of the cones onto $H_0$ are simplices and these simplices are dilated and translated copies of $\triangle$ and the only point of intersection of the simplices is the projection of the max-sum. Hence, we obtain $d_{\triangle}(O,m)=d_{\triangle}(p,m)=||p \oplus O||_1/(n+1)$. 
\end{proof}

\begin{lemma}\label{tropcut_lem}
For a subset $S$ of the rows of the Laplacian matrix of the graph, let $u_S=\sum_{b_i \in S}b_i$. The $\ell_1$-norm of the max-sum of $u_S$ and $O$ is the size of the cut defined by $S$.
\end{lemma}
\begin{proof}
The max-sum $u_S \oplus O$ of $u_S$ and $O$ is given by:\\
\begin{equation}
{(u_S \oplus O)}_i=
\begin{cases}
 {u_S}_i, \text{ if $(u_S)_i>0$},\\
 0, \text{ otherwise.}
\end{cases}
\end{equation}
Hence, we consider the positive coordinates of $u_S$. Now note that since the sum of coordinates of $u_S$ is zero, the absolute sum of the positive coordinates is equal to absolute sum of negative coordinates. Furthermore, the negative coordinates are characterized by vertices that do not belong to the cut and that the sum of the absolute values of the negative-valued coordinates of $u_S$ is the size of the cut $S$.
\end{proof}

\begin{theorem}\label{pacrad_theo} The packing radius of $L_G$ under the simplicial distance function $d_{\triangle}$ is equal to  $\frac{\text{MC}_1(G)}{n+1}$ where MC$_1(G)$ is the size of the $\ell_1$-minimum cut. \end{theorem}
\begin{proof}
First, observe that the lattice point that defines the packing radius under the distance function $d_{\triangle}$ is a Voronoi neighbour of the origin under $d_{\triangle}$. Hence, we only restrict to the Voronoi neighbours. Now for the case of general connected graphs, the characterization of the Voronoi neighbours of the origin obtained in Theorem \ref{VorNeigh_theo} is not explicit enough. Hence, we perform the standard perturbation of the Laplacian matrix (see Subsection \ref{VorLap_Subsect}) and consider lattices generated by these perturbed matrices. Using Lemma \ref{pert_lem}, we know that the 
packing radius of the perturbed lattice $L^{\epsilon}_G$ is defined by a point of the form
$u_S=\sum_{i \in S}b^{\epsilon}_i$ for some non-trivial subset $S$ of $V(G)$ and 
$b^{\epsilon}_i$ being the perturbed vector of $b_i$. Using the fact that the packing radius is preserved under perturbation (see Appendix Subsection \ref{Laqseq_subsect}), we deduce that the packing radius is defined by a point of the form $u_S$ for a non-trivial subset $S$ of $V(G)$. By Lemma \ref{radtrop_lem} and the definition of packing radius we have Pac$_{\triangle}(L_G)=\min_{S}||u_S \oplus O||_1/{n+1}$ where the minimum is taken over all the non-trivial cuts $S$. We now use Lemma \ref{tropcut_lem} to deduce that $\min_{S}||u_S \oplus O||_1$ is equal to the size of the minimum cut of $G$.
\end{proof}

\begin{theorem} (Section 8.3 of \cite{AmiMan10})\label{coverrad_theo}
The covering radius of the Laplacian lattice is equal to $\frac{g+n}{n+1}$.
\end{theorem}

\section{The Shortest Vector of the Laplacian lattice}\label{Short_Sect}

 We will provide a combinatorial interpretation of the norm of the shortest vector 
 of the Laplacian lattice under the simplicial distance function $d_{\triangle}$. We will see
 that the norm of the shortest vector of the Laplacian lattice under the simplicial distance function $d_{\triangle}$ is in fact a certain variant of the minimum cut in the graph. A precise definition follows:

\begin{definition} {\bf ($\ell_{\infty}$-Minimum Cut)}\label{ellinfcut_def}
For a non-trivial cut $S$ of $V(G)$ i.e. $S$ is neither empty nor equal to $V(G)$, define the weight of the cut as
$\mu_{\infty}(S)=\max\{deg_{S,\bar{S}}(v)|$ $v \in \bar{S}\}$ where $deg_{S,\bar{S}}(v)$ is the degree of the vertex $v$
across the cut $S$. Now define the $\ell_{\infty}$-minimum cut MC$_{\infty}(G)$ as the minimum of $\mu_{\infty}(S)$ over all non-trivial cuts $S$.
\end{definition}

Note that for a simple connected graph $G$, we have MC$_{\infty}(G)=1$.






\begin{theorem}
The length of the shortest vector $\nu_{\triangle}(L_G)$ of the Laplacian 
lattice under the simplicial distance function $d_{\triangle}$ 
is equal to $\mathrm{MC}_{\infty}(G)$.
\end{theorem}
\begin{proof}
First let us consider the case where $G$ is a multigraph where every pair of vertices are connected by an edge.
By Lemma \ref{shortvec_lem}, we know that every shortest vector in the distance function $d_{\triangle}$
must be a Voronoi neighbour of the origin under $d_{\triangle}$. 
We know that the Voronoi neighbours of the origin 
under $d_{\triangle}$ are of the form:  
$u_{S}=\sum_{i \in S}b_i$ for some non-trivial subset $S$ of $V(G)$. 
For the case of general connected graphs, the characterization of the 
Voronoi neighbours of the origin obtained in Theorem \ref{VorNeigh_theo} is not explicit enough. 
Hence, we perform the standard perturbation (see Subsection \ref{VorLap_Subsect}) of the Laplacian matrix 
and consider lattices generated by these perturbed matrices. By Lemma \ref{pert_lem}, we know that the 
shortest vector of the perturbed lattice $L^{\epsilon}_G$ is defined by a point of the form
$u_S=\sum_{i \in S}b^{\epsilon}_i$ for some non-trivial subset $S$ of $V(G)$ and 
$b^{\epsilon}_i$ being the perturbed vector of $b_i$.
Using the fact that the quantity $\nu_{\triangle}(.)$ is preserved 
as the perturbation tends to zero (See Appendix Section \ref{Laqseq_subsect}), 
we deduce that the shortest vector for the general lattice must be of the form: 
$u_{S}=\sum_{i \in S}b_i$ for some non-trivial subset $S$ of $V(G)$.

Consider a subset $S$ such that $u_{S}$ is a shortest vector.
First, recall that $d_{\triangle}(O,u_{S})=|\min_{j}~{u_{S}}_j|$.
Now, since $u_S \neq O$ only the negative coordinates of $u_S$ define $d_{\triangle}(O,u_{S})$
and the negative coordinates are indices $j$ such that $v_j \notin S$.
Hence $d_{\triangle}(O,u_{S})=\max\{deg_{S,\bar{S}}(v)$ $| v \in \bar{S}\}$.
We have: 
\begin{equation}
\nu_{\triangle}(L_G)= \min_{S}\{d_{\triangle}(O,u_{S})\}=\\
\min_{S}\max\{deg_{S,\bar{S}}(v) ~| v \in \bar{S}\}=\mathrm{MC}_{\infty}(G).
\end{equation}

\end{proof}

\begin{corollary} For a simple connected graph $G$, we have $\nu_{\triangle}(L_G)=1$.\end{corollary}

We will now apply the combinatorial interpretation of the parameters of the Laplacian lattice that we have gained in the previous sections to answer some natural questions that arise from the correspondence between the graph and its Laplacian lattice.

\section{The Delaunay polytope of the Laplacian lattice 
          characterizes the graph completely up to isomorphism}\label{Del_Sect}


A natural question that arises with the correspondence between the Laplacian lattice and a graph is whether
the Laplacian lattice characterizes the underlying graph completely up to isomorphism? The first observation towards answering this question is the following lemma:

\begin{lemma}\label{Laptree_lem} The Laplacian lattice of any tree on $n+1$ vertices 
is the root lattice $A_n$. \end{lemma}

\begin{proof}
We know that the Laplacian lattice is a sublattice of the root lattice $A_n$ and 
the covolume of a Laplacian lattice with respect to $A_n$ is equal to the number of spanning trees of
the graph. This implies that in the case of trees the covolume of the Laplacian lattice is
equal to one. Hence, the Laplacian lattice of a tree is the root lattice $A_n$ itself. 
\end{proof}

Lemma \ref{Laptree_lem} shows that the Laplacian lattice itself does not characterize a 
graph completely up to isomorphism. However, we will now see that the 
Delaunay triangulations of the Laplacian lattice $L$ with respect to 
the simplicial distance function provide more refined 
information about the graphs that have $L$ as their Laplacian lattice. 
More precisely, each graph provides a Delaunay triangulation of its 
Laplacian lattice and the Delaunay polytope of the origin characterizes 
the graph completely up to isomorphism. In the course of showing this result,
 we also study the structure of the Delaunay polytope. As we noted, in the case of graphs with no zero entries in the Laplacian matrix our study of the Delaunay triangulation is in fact a study of the Scarf complex of the Laplacian lattice and in general, the Delaunay triangulation contains the Scarf complex of $L_G$.


\subsection{The structure of the polytope $H_{Del_{G}}(O)$}






We know from Theorem \ref{LapDel_theo} that each graph provides a Delaunay triangulation of its 
Laplacian lattice in the following manner: Recall that $\{b_0,\dots,b_{n}\}$ are the rows of the Laplacian matrix of $G$. 
For each permutation $\sigma \in S_{n+1}$, we define $u^{\sigma}_i=\sum_{j=0}^{i}b_{\sigma(j)}$ for $i$ from $0$ to $n$. Note that $u^{\pi}_n=(0,\dots,0)$ for every permutation $\pi \in S_{n+1}$. We define the simplex $\triangle_{\sigma}$ as the convex hull of $u^{\sigma}_0,\dots,u^{\sigma}_{n}$. The Delaunay polytope of the origin i.e. the set of Delaunay simplices with the origin as a vertex, 
is given by $\{\triangle_{\sigma}\}_{\sigma \in S_{n+1}}$. See Figure \ref{Del_fig} for the Delaunay polytope of 
some small graphs. We denote this polytope by $H_{Del_G}(O)$. 
We know from Theorem \ref{LapDel_theo} that each graph provides a Delaunay triangulation of its 
Laplacian lattice in the following manner: Recall that $\{b_0,\dots,b_{n}\}$ are the rows of the Laplacian matrix of $G$. 
For each permutation $\sigma \in S_{n+1}$, we define $u^{\sigma}_i=\sum_{j=0}^{i}b_{\sigma(j)}$ for $i$ from $0$ to $n$. Note that $u^{\pi}_n=(0,\dots,0)$ for every permutation $\pi \in S_{n+1}$. We define the simplex $\triangle_{\sigma}$ as the convex hull of $u^{\sigma}_0,\dots,u^{\sigma}_{n}$. The Delaunay polytope of the origin i.e. the set of Delaunay simplices with the origin as a vertex, 
is given by $\{\triangle_{\sigma}\}_{\sigma \in S_{n+1}}$. See Figure \ref{Del_fig} for the Delaunay polytope of 
some small graphs. We denote this polytope by $H_{Del_G}(O)$. 

We first describe the vertex set of $H_{Del_{G}(O)}$. A study of low-dimensional examples leads us to the claim 
that every vertex of $H_{Del_{G}}(O)$ is of the form $u^{\sigma}_k$ for some integer $k$ from $0$ to $n-1$
and a permutation $\sigma \in S_{n+1}$. 
In order to show this claim, we proceed as follows: we consider the convex hull, 
call it $H'(G)$, of points $u^{\sigma}_k$ for $k$ from $0$ to $n$ and over all permutations 
$\sigma \in S_{n+1}$. We then show that every point of the form $u^{\sigma}_k$ for $k$ 
from $0$ to $n-1$ is a vertex of $H'(G)$ and that $H'(G)=H_{Del_{G}}(O)$.
\begin{figure}
  \begin{center}
    \includegraphics[height=6cm]{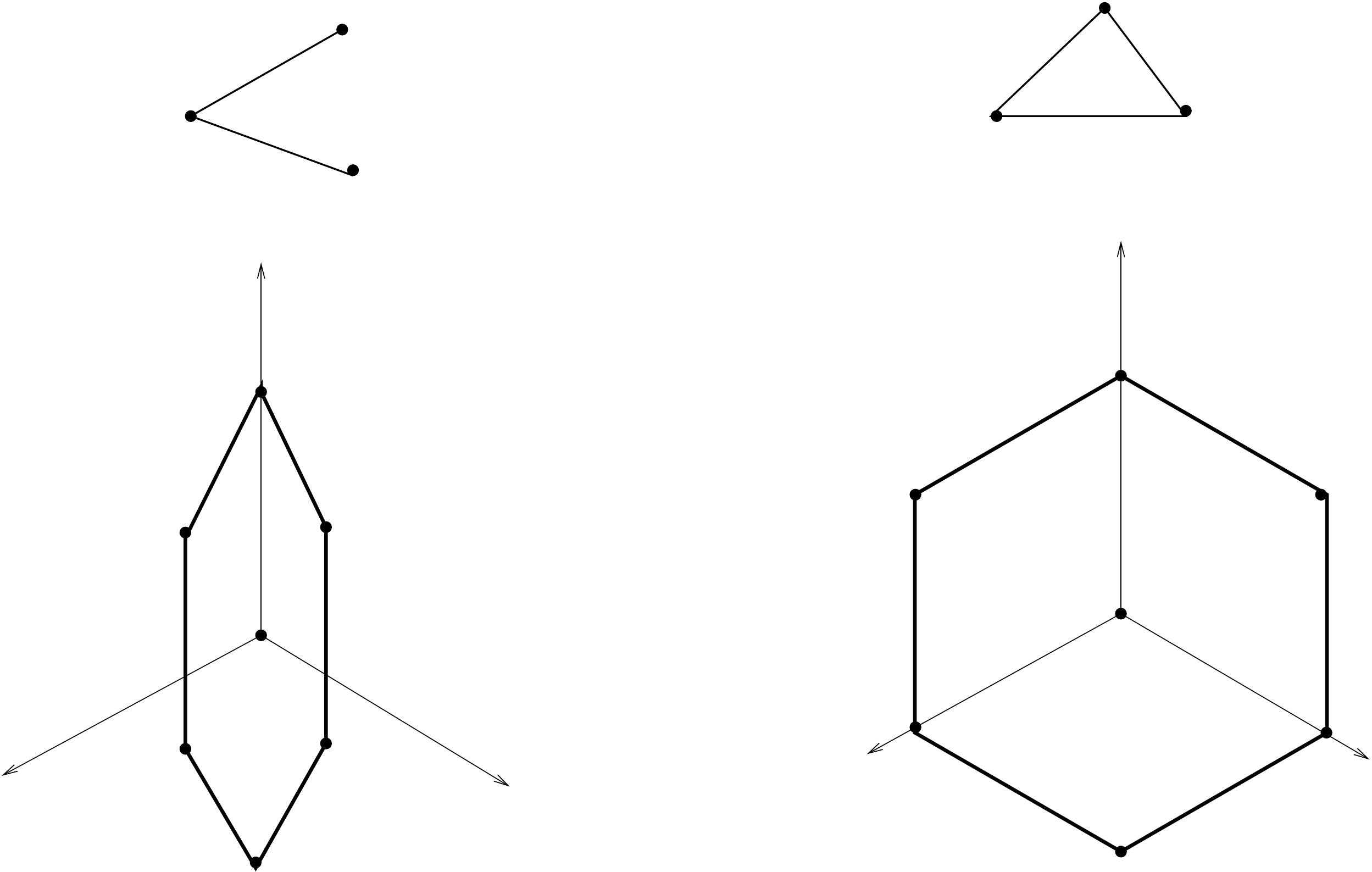} 
  \end{center}
\caption{ The Delaunay polytope of a path on two edges (left) and a triangle.}\label{Del_fig}
 \end{figure}

\begin{lemma}\label{verchar_lem} Every point of the form $u^{\sigma}_k$ is a vertex of $H'(G)$ where $k$ varies from $0$ to $n-1$ and  $\sigma \in S_{n+1}$. \end{lemma}
\begin{proof}
We show that for every point of the form $u^{\sigma}_k=\sum_{i=0}^{k}b_\sigma(i)$ for $k$ from $0$ to $n-1$
there exists a point $w$ such that the linear functional: $f(x)=w \cdot x^{t}$ has the property:

$f(b_{\sigma(i)})=
\begin{cases}
> 0 \text{ if $0 \leq i \leq k$,}\\
< 0 \text{ otherwise.}
\end{cases}
$

The details of the construction of the functional $f$ are as follows:
Consider the set $C_{\sigma,k}$  of points $(p_0,\dots,p_{n}) \in H_0$ such that

$p_{\sigma(i)}=
\begin{cases}
> 0 \text{ if $0 \leq i \leq k$,}\\
< 0 \text{ otherwise.}
\end{cases}
$

We make the following observations:

1. $C_{\sigma,k}$ is a cone.

2. $C_{\sigma,k}$ is not empty for $k$ from $0$ to $n-1$.

Take a point $q$ in $C_{\sigma,k}$. Since $G$ is connected, $L_G$ is a lattice of 
dimension $n$ and a basis of $L_G$ spans $H_0$. Consider the basis $\{b_0,\dots,b_{n-1}\}$ (the first $n$ rows of $Q(G)$). 
Let $q=\sum_{i=0}^{n-1}v_i \cdot b_i$. Set $w_i=v_i$ for $i$ from $0$ to $n-1$ and $w_n=0$. Using the symmetry of the Laplacian matrix, we can easily verify that the functional $f$ has the desired properties. By the properties of $f$, it follows that $u^{\sigma}_k$ is the unique maximum of $f(x)$ among vertices of the form $u^{\pi}_j$ for an arbitrary permutation $\pi \in S_{n+1}$. 


This implies that $u^{\sigma}_k$ is also the unique maximum of $f(x)$ over the polytope $H'(G)$. 
Using standard arguments in linear optimization \cite{Schrijver86}, this implies that $u^{\sigma}_k$ is a vertex of $H'(G)$.
\end{proof}


Next we characterize points in the simplex $\triangle_{\sigma}$:

\begin{lemma} \label{simpchar_lem} A point $p$ is contained in $\triangle_{\sigma}$  if and only if $p$ can be written as 
$p=\sum_{i=0}^{n}\lambda_{i}b_{\sigma(i)}$ where $1=\lambda_0\geq \lambda_1 \geq \dots \geq\lambda_{n} \geq 0$.
\end{lemma}
\begin{proof}
If $p$ is contained in $\triangle_{\sigma}$ then, we can write: $p=\sum_{i=0}^{n} \mu_i u^{\sigma}_i$ where $\mu_i \geq 0$ and $\sum_{i=0}^{n}\mu_i=1$. We can plug in $u^{\sigma}_k=\sum_{i=0}^{k}b_{\sigma(i)}$ to the equation to obtain: $p=\sum_{i=0}^{n}\lambda_ib_{\sigma(i)}$ where $\lambda_k=\sum_{i=k}^{n}\mu_i$. Now, observe that  $1=\lambda_0\geq \lambda_1 \geq \dots \geq\lambda_{n} \geq 0$.

Conversely, if a point can be written as $p=\sum_{i=0}^{n}\lambda_i b_{\sigma(i)}$ where $1=\lambda_0\geq \lambda_1 \geq \dots \geq\lambda_{n} \geq 0$.
We set $\mu_i=\lambda_i-\lambda_{i+1}$ for $0 \leq i \leq n-1$ and $\mu_{n}=\lambda_{n}$ and we have $p=\sum_{i=0}^{n}\mu_iu^{\sigma}_i$. We finally
verify the following properties: i. $\mu_i \geq 0$, since $\lambda_i \geq \lambda_{i+1}$ for $i$ from $0$ to $n-1$, ii. $\lambda_{n} \geq 0$ and iii. $\sum_{i=0}^{n}\mu_i=\lambda_0=1$. This shows that $p$ is contained in $\triangle_{\sigma}$.
\end{proof}
 
\begin{corollary}\label{conv_corr} The set $H_{Del_{G_1}}(O)$ is convex.\end{corollary}
\begin{proof}
Consider points $p_1$ and $p_2$ in simplices $\triangle_{\sigma_1}$ and $\triangle_{\sigma_2}$ for some $\sigma_1,\sigma_2 \in S_{n+1}$. By Lemma \ref{simpchar_lem},
we can write: $p_1=\sum_{i=0}^{n}\lambda^{1}_{i}b_{\sigma_1(i)}$ and $p_2=\sum_{i=0}^{n}\lambda^{2}_{i}b_{\sigma_2(i)}$ where $1=\lambda^{i}_0\geq \lambda^{i}_1 \geq \dots \geq\lambda^{i}_{n} \geq 0$ for $i\in \{1,2\}$. Each point $p$ in the line segment joining $p_1$ and $p_2$ can be written as
$\mu p_1 + (1-\mu )p_2$ for some $0 \leq \mu \leq 1$. We can write: $p=\sum_{i=0}^{n}\lambda_ib_i$ and there exists a permutation $\sigma$ such that 
$1 \geq \lambda_{\sigma(0)}\geq \lambda_{\sigma(2)}\geq \dots \geq\lambda_{\sigma(n)} \geq 0$. We add $O=(1-\lambda_{\sigma(0)})\sum_{i=0}^{n}b_i$ to the right hand side of the equation to obtain a point of the form stated in Lemma \ref{simpchar_lem}. This concludes the proof.
\end{proof}

\begin{lemma} For any undirected connected graph $G$, we have $H'(G)=H_{Del_{G}(O)}$.\end{lemma}
\begin{proof}
Consider a point $p$ in $H_{Del_{G}}(O)$. By definition, $p$ belongs to some simplex of the form $\triangle_{\sigma}$. Now, observe that all the vertices of $\triangle_{\sigma}$ except the origin are contained in $H'(G)$. Now, let $p=\sum_{i=0}^{n-1}\lambda_i u^{\sigma}_i+\lambda_{n} \cdot O$ for some $\lambda_i \geq 0$ and $\sum_{i=0}^{n}\lambda_i=1$. We have $p=\sum_{i=0}^{n-1}\lambda_i u^{\sigma}_i+\lambda_{n}\cdot O=\sum_{i=0}^{n-1}\lambda_i u^{\sigma}_i+\lambda_{n}(b_0 + \sum_{j=1}^{n}b_j)/2$.
This shows that $p$ can be written as a convex combination of points in $H'(G)$ and hence $p$ is contained in $H'(G)$.

We show the converse by contradiction. Assume that there exists a point $p$ in $H'(G)/H_{Del_{G}(O)}$. By Corollary \ref{conv_corr}, $H_{Del_{G}(O)}$ is a convex polytope and hence, a closed subset of $H_0$ equipped with the Euclidean topology. 
By the separation theorem for closed convex sets \cite{Schneider93}, there exists a linear functional $f(x)$ such that $f(p)<0$ and $f(y)>0$ for all $y \in H_{Del_{G}(O)}$. Since $u^{\sigma}_k$ is contained in $H_{Del_{G}}(O)$, we know that $f(u^{\sigma}_k)>0$ for all $0 \leq k \leq n$ and $\sigma \in S_{n+1}$. But since, $p$ is contained in $H'(G)$ it can be written as a convex combination of the vertices of $H'(G)$ and the vertices of $H'(G)$ are points of the form $u^{\sigma}_k$ for $k$ from $0$ to $n-1$ and hence $f(p)>0$. We obtain a contradiction.
\end{proof}
 
\begin{corollary}The set $H_{Del_{G}}(O)$ is a convex polytope and has $2^{n+1}-2$ vertices.\end{corollary}

We now describe the facet structure of $H_{Del_{G}}(O)$. We know that every vertex of $H_{Del_{G}}(O)$ is of the form:

\begin{equation}\label{sum_eq} v=\sum_{j=0}^{k}b_{i_j}, \text { for $k$ from $1$ to $n$} \text{ and $b_{i_j}$s are all distinct}.\end{equation}

We define the set $V_i$, for $i$ from $0$ to $n$ as the subset of vertices that contain $b_i$ in their representation of the form stated in (\ref{sum_eq}). Define the set $F_{i,j}=V_i-V_j.$
\begin{lemma}\label{aff_lem} For each integer $0 \leq i,j \leq n$ and $i \neq j$, the affine hull of the elements of $F_{i,j}$ is a $n-1$-dimensional affine space.\end{lemma}
\begin{proof}
For the sake of convenience, we consider $F_{0,n}$. The set $S=\{b_0,b_0+b_1,\dots, b_0+b_1+b_2+\dots +b_{n-1}\}$ is contained in $F_{0,n}$. Since $G$ is  connected, the elements of $S$ are linearly independent and hence their affine hull is a $n-1$-dimensional affine space. 

Now, we show that every element of $F_{0,n}$ is contained in the affine hull of $S$. The affine hull of $S$ consists of elements of the form $\sum_{i=0}^{n}\alpha_iu_i$ where $u_i=\sum_{k=0}^{i}b_k$ and $\sum_{i=0}^{n}\alpha_i=1$.
This can be written as $\sum_{i=0}^{n}\lambda_ib_i$ where $\lambda_i=\sum_{j=i}^{n}\alpha_j$.  This means that $\lambda=L\alpha$ where $\lambda=(\lambda_0,\dots,\lambda_n)$, $\alpha=(\alpha_0,\dots,\alpha_n)$ and $L$ is the lower triangular matrix with unit entries.
By definition, the elements of $F_{0,n}$ are of the form:
$b_0+b_{i_1}+\dots+b_{i_k}$ where $i_j \notin \{0,n\}$.
Now consider an element $u=b_0+b_{j_1}+\dots+b_{j_k}$ in $F_{0,n+1}$ and let $S'=\{0,j_1,\dots,j_k\}$. Set $\lambda'=i_{S'}$, the indicator vector of the set $S'$ and since $L$ is invertible there exists a unique vector $\alpha'=L^{-1}\lambda'$ and we verify that $\sum_{i=0}^{n}\alpha'_i$ is the coefficient of $b_0$ in $u$ and is hence equal to 1. 
This shows that $u$ is contained in the affine hull of $S$. The same argument can be done for any $F_{i,j}$ by replacing $0$ by $i$ and $n$ by $j$. 
\end{proof}

\begin{lemma} \label{facecons_lem}For each integer $0 \leq i,j \leq n$ and $i \neq j$, the convex hull of the elements of $F_{i,j}$ is a facet of $H_{Del_{G}}(O)$. Furthermore, each element of $F_{i,j}$ is in convex position.\end{lemma}
\begin{proof}
For the sake of convenience, we consider $F_{0,n}$. By Lemma \ref{aff_lem}, the affine hull of the elements of $F_{0,n}$ spans a $n-1$-dimensional space $K$ and we know that $H_{Del_{G}}(O)$ is a $n$-dimensional polytope. In order to show that $F_{0,n}$ is a facet of $H_{Del_{G}}(O)$ we need to show that the affine hull of the elements of $F_{0,n}$ supports $H_{Del_{G}}(O)$. This can be seen as follows: since $G$ is connected we know that $\{b_0,\dots,b_{n-1}\}$ is a basis of $H_0$. Hence, every point in $H_0$ can be uniquely written as $\sum_{i=0}^{n-1}\alpha_ib_i$ for some $\alpha_i \in \mathbb{R}$ and the hyperplane $K$ is given by $\alpha_0=1$. Now, consider a vertex $v$ of $H_{Del_{G}}(O)$ not contained in $F_{0,n}$. Let $v=\sum_{j=1}^{k}b_{i_j}$. Since $v$ is not contained in $F_{0,n}$, either $b_0$ is not contained in the sum representing $v$ or $b_{n}$ is contained in the sum. If $b_0$ is not contained in the sum and $b_{n}$ is contained in the sum then $\alpha_0=-1$. If $b_0$ and $b_{n}$ are contained in the sum then $\alpha_1=0$ and also, if $b_0$ and $b_{n}$ are not contained in the sum then $\alpha_0=0$. This shows that all the vertices of $H_{Del_{G}}(O)$ that are not in $F_{0,n-1}$ are strictly contained in $K^{-}$ i.e. the halfspace $\alpha_0 \leq 1$. This suffices to conclude that $K \cap H_{Del_{G}}(O)=F_{0,n}$. 
By Lemma \ref{verchar_lem}, all the elements of $F_{0,n}$ are in convex position 
 and that concludes the proof.
\end{proof}

\begin{lemma}\label{nootherfacets_lem} The facets of $H_{Del_{G}}(O)$ are exactly of the form $F_{i,j}$ for $0 \leq i,j \leq n$ and $i \neq j$.\end{lemma}
\begin{proof} By construction, the facets of $H_{Del_{G}}(O)$ must be contained in the affine hull of the facets of $\triangle_{\sigma}$. 
By Lemma \ref{facecons_lem}, we know that the affine hull of any facet of $\triangle_{\sigma}$ not containing the origin contains a facet of $H_{Del_{G}}(O)$.
Now, it suffices to show that the affine hull of a facet of $\triangle_{\sigma}$ containing the origin does not contain a facet of $H_{Del_{G}}(O)$.
Consider a facet $F$ of $\triangle_{\sigma}$. We may assume, without loss of generality, that $F$ contains all vertices of $\triangle_{\sigma}$ apart from  $u_{k}^{\sigma}$ for some $0 \leq k \leq n-1$.
Assume that $F$ is a facet of $H_{Del_{G}}(O)$.
This means that there is an affine function $f$ such that
    
P1. $f({u^{\sigma}_i})=c$ for all $i \neq k$. 

P2. $f(v)>c$ for all vertices of $H_{Del_{G}}(O)$ that are not contained in the affine hull of $F$.\\
Note that by the property P1, we have $f(O)=c$. Consider the linear function $g(x)=f(x)-f(O)$. We have $g({u^{\sigma}_i})=0$ for all $i \neq k$. Now, suppose that $g(u^{\sigma}_k)=0$, then $f({u^{\sigma}_k})=f({u^{\sigma}_j})$ for all $j \neq k$, but since $G$ is connected, $u^{\sigma}_k$ is not contained in the affine hull of $F$ and  this contradicts the property P2. If $g(u_{\sigma(k)})<0$, then $f({u^{\sigma}_k})<c$ and hence again contradicts property P2. On the other hand, if $g(u_{\sigma(k)})>0$, then $g(-u_{\sigma(k)})<0$ and $f({-u^{\sigma}_k})<c$ and $-u^{\sigma}_k$ is indeed a vertex of $H_{Del_{G}}(O)$. This again contradicts the property P2. This concludes the proof.
\end{proof}

\begin{corollary} The number of facets of $H_{Del_{G}}(O)$ is $n \cdot (n+1)$. \end{corollary}

We now characterize the edges of $H_{Del_{G}}(O)$.

\begin{lemma}\label{edgechar_lem}
The edges of $H_{Del_{G}}(O)$ are of the form: $e^{\sigma}_k=(u^{\sigma}_k,u^{\sigma}_{k+1})$ for $k$ from $0$ to $n-1$ and $\sigma \in S_{n+1}$.
\end{lemma}
\begin{proof}

The proof is modelled similarly as the proof of Lemma \ref{verchar_lem}.
We show that for each $e^{\sigma}_k$, there exists a linear function $f(x)=w \cdot x^t$ such that
$f(u^{\sigma}_k)=f(u^{\sigma}_{k+1})=c$ and $f(v)<c$ for all other vertices $v$ of $H_{Del_{G}}(O)$. 

Consider the set $C_{e^{\sigma}_k}$ of points $(p_1,\dots,p_{n+1}) \in H_0$ such that 

$p_{\sigma(i)}=
\begin{cases}
> 0 \text{ if $0 \leq i \leq k$,}\\
= 0 \text{ if $i=k+1$,}\\
<0, \text {otherwise}.
\end{cases}
$

We make the following observations:

1. $C_{e^{\sigma}_k}$ is a cone.

2. $C_{e^{\sigma}_k}$ for $0 \leq k \leq n-1$ is not empty.

Take a point $q$ in $C_{\sigma,k}$. Since $G$ is connected, $L_G$ is a lattice of 
dimension $n$ and a basis of $L_G$ spans $H_0$. Consider the basis $\{b_0,\dots,b_{n-1}\}$ (the first $n$ rows of $Q(G)$)
and write $q=\sum_{i=0}^{n-1}v_i \cdot b_i$. Set $w_i=v_i$ for $i$ from $0$ to $n-1$ and $w_n=0$.

The functional $f(x)=w \cdot x^{t}$ attains its maximum precisely on the edge $e^{\sigma}_k$.
Using standard arguments in linear optimization \cite{Schrijver86}, this implies that $e^{\sigma}_k$ is an edge of $H'(G)$.

Now it remains to show that there are no other edges of $H_{Del_{G}}(O)$. By construction of $H_{Del_{G}}(O)$, an edge of $H_{Del_{G}}(O)$ is contained on the affine hull of an edge of $\triangle^{\sigma}$ for some $\sigma \in S_{n+1}$. An edge of $\triangle^{\sigma}$  is of the form $(u^{\sigma}_{i},u^{\sigma}_{j})$ for some $0 \leq i,j \leq n$ and $i \neq j$.  We know that $(u^{\sigma}_k,u^{\sigma}_{k+1})$ is an edge for $0 \leq k \leq n-1$ and $\sigma \in S_{n+1}$. Now, consider an edge $e_{i,j}$ of $\triangle^{\sigma}$ of the form $(u^{\sigma}_i,u^{\sigma}_{j})$ where $j \neq i+1$ and $i<j$. Assume that the affine hull of $e_{i,j}$ contains an edge of $H_{Del_{G}}(O)$. 
This implies that there exists an affine function $f$ such that
P1. $f(u^{\sigma}_i)=f(u^{\sigma}_j)=c$.
P2. $f(v) \leq c$ for all other vertices of $H_{Del_{G}}(O)$ not contained in the affine hull of $e_{i,j}$.
Let $g$ be the linear function $g(x)=f(x)-f(O)$. Now, by property P1, we have  $f(u^{\sigma}_j-u^{\sigma}_i)=g(u^{\sigma}_j-u^{\sigma}_i)=0$. This means that 
$f(b_{\sigma(j)}+\dots+b_{\sigma(i+1)})=g(b_{\sigma(j)})+\dots+g(b_{\sigma(i+1)})=0$. Now, either $g(b_{\sigma(j)})=g(b_{\sigma(j-1)})=\dots=g(b_{\sigma(i+1)})=0$ or there exists $j \leq k_1, k_2 \leq i+1$ such that $g(b_{\sigma(k_1)})>0$ and $g(b_{\sigma(k_2)})<0$. In the first case, we have $f(u^{\sigma}_{i+1})=f(u^{\sigma}_i)+g(b_{\sigma(i+1)})=c$ but by property P2, this means that $u^{\sigma}_{i+1}$ is contained in the affine hull of $e_{i,j}$ and hence $u^{\sigma}_{i+1}=\lambda \cdot (u^{\sigma}_{i}-u^{\sigma}_{j})$ for some $\lambda \in \mathbb{R}$.
But this contradicts the connectivity of $G$. In the second case, we have $f(u_{\sigma(i)}+b_{\sigma(k_1)})=f(u^{\sigma}_i)+g(b_{\sigma(k_1)})>c$ and since $u_{\sigma(i)}+b_{\sigma(k_1)}$ is a vertex of $H_{Del_{G}}(O)$, we obtain a contradiction. This concludes the proof. 
\end{proof}

\begin{problem}
We have characterized the zero, one and $n-1$ dimensional faces of $H_{Del_{G}}(O)$. Can we obtain similar characterizations of the other faces of $H_{Del_{G}}(O)$?
\end{problem}

\subsection{Combinatorics of $H_{Del_{G}}(O)$.}

Recall that the $f$-vector of an $n$-dimensional polytope $P$ is the vector $(f_0,\dots,f_{n-1})$ where $f_k$ is the number of $k$-dimensional faces of $P$. In the previous section, we showed that $f_0(H_{Del_{G}}(O))=2^{n+1}-2$ and $f_{n-1}(H_{Del_{G}}(O))=n \cdot (n+1)$.

For each vertex $v$, we denote by $d_k(v)$, the number of $k$-dimensional faces incident on $v$.

\begin{lemma}\label{degcomp_lem} For a vertex $v$ of the form $\sum_{j=0}^{k-1}b_{i_j}$ where the indices $i_j$s are all distinct, we have $d_{n-1}(v)=k \cdot (n+1-k)$.\end{lemma}
\begin{proof}
Consider a vertex $v$ of the form $\sum_{j=0}^{k-1}b_{i_j}$ where $i_j$s are all distinct and let $S_v=\{i_0,\dots,i_{k-1}\}$. The facets that contain $v$ are of the form:
$F_{i,j}$ where $i \in S_v$ and $j \notin S_v$. There are $k \cdot (n+1-k)$ such facets.
\end{proof}

\begin{corollary}\label{symm_cor} Let $g(k)=k \cdot (n+1-k)$, if $g(k_1)=g(k_2)$ for $k_1 \neq k_2$ then $k_1+k_2=n+1$. \end{corollary}

 Recall that the polytopes $\mathcal{P}_1$ and $\mathcal{P}_2$ in $H_0$ are congruent if there exists an isometry $M$ such that $\mathcal{P}_2=M(\mathcal{P}_1)$.
By an isometry, we mean a map of the form $M \cdot x=A \cdot x+t$ where $A$ is an orthogonal transformation and $t \in H_0$ 

\subsection{Correspondence between $G$ and $H_{Del(G)}(O)$}

For a convex polytope $\mathcal{P}$ in $H_0$, let $Aut(\mathcal{P})$ denote the automorphism group of $\mathcal{P}$. 
The following simple lemmas turn out to be useful.

\begin{lemma}\label{perm_lem} Let $\mathcal{P}$ be a polytope in $\mathbb{R}^n$, let $M$ be an element of $Aut(\mathcal{P})$, then $M$ permutes the vertices of $\mathcal{P}$.\end{lemma}


A simple calculation shows the following:

\begin{lemma}\label{autorth_lem} Every element of $Aut(\triangle)$ is an orthogonal transformation i.e. the translation part of the isometric map is zero.\end{lemma}




\begin{theorem}\label{geomcomb_theo}Let $G_1$ and $G_2$ be undirected connected graphs. The polytopes $H_{Del_{G_1}}(O)$ and $H_{Del_{G_2}}(O)$ are
 congruent if and only if $G_1$ and $G_2$ are isomorphic.\end{theorem}

\begin{proof}
If $G_1$ and $G_2$ are isomorphic, then we know that there exists a permutation matrix or in other words $\sigma \in Aut(S)$ such that $Q(G_2)=\sigma \cdot Q(G_1)\cdot \sigma^{-1}$.
This implies that we have $b^{G_2}_i=\sigma \cdot b^{G_1}_{\sigma^{-1}(i)}$. We now claim that $H_{Del_{G_2}}(O)=\sigma.H_{Del_{G_1}}(O)$. In order to see this observe that $\triangle_{\sigma_1}(G_2)=\sigma \cdot \triangle_{\sigma^{-1}\sigma_1}(G_1)$. By the definition of $H_{Del_{G}}(O)$, we have $H_{Del_{G_2}}(O)=\cup_{\sigma_1 \in S_{n+1}}\triangle_{\sigma_1}(G_2)=\cup_{\sigma_1 \in S_{n+1}}\sigma \triangle_{\sigma^{-1}\sigma_1}(G_1)=\sigma.H_{Del_{G_1}}(O)$.
By Lemma \ref{autorth_lem}, we know that $\sigma$ is an orthogonal transformation and hence this implies that $H_{Del_{G_1}}(O)$ and $H_{Del_{G_2}}(O)$ are congruent.

Conversely, if $H_{Del_{G_1}}(O)$ and $H_{Del_{G_2}}(O)$ are congruent then there exists an isometry $M(x)=A(x)+t$ for some  orthogonal transformation $A$ and $t \in H_0$ such that $H_{Del_{G_2}}(O)=M(H_{Del_{G_1}}(O))$. Since, $A$ is a non-singular trasformation, we know that $M$ induces a bijection between the facets of $H_{Del_{G_1}}(O)$ and $H_{Del_{G_2}}(O)$. We also know that a facet of $H_{Del_{G}}(O)$ is of the form $F^{G}_{i,j}$ for some $0 \leq i,j \leq n$. Now, consider an arbitrary facet $F^{G_2}_{i,j}$ of $H_{Del_{G_2}}(O)$ and let the facet $M(F^{G_2}_{i,j})$ of $H_{Del_{G_1}}(O)$ be $F^{G_1}_{i',j'}$. Furthermore, we know that $M$ induces a bijection between the vertices of $F^{G_2}_{i,j}$ and the vertices of $F^{G_1}_{i',j'}$ and that the $n-1$-th degree $d_{n-1}$ is conserved and by Corollary \ref{symm_cor} this means that 
either $b^{G_2}_{i}=M(\sum_{l=0}^{n}b^{G_1}_{l}-b^{G_1}_{i'})=M(-b^{G_1}_{i'})$ or $b^{G_2}_{i}=M(b^{G_1}_{i'})$. 
In the first case, consider the map $M'(x)=-M(x)$. Indeed $M'$ is an orthogonal transformation and since $H_{Del_{G_1}}(O)$ is a centrally symmetric polytope, 
we have $H_{Del_{G_2}}(O)=M(H_{Del_{G_1}}(O))$. Hence, we may assume without loss of generality that $b^{G_2}_{i}=M(b^{G_1}_{i'})$. By Corollary \ref{symm_cor}, this implies that $b^{G_2}_j=M(b^{G_1}_{j'})$. Now, we know that for every edge $e$ of facet $F^{G_2}_{i,j}$, $M(e)$ is an edge incident on $F^{G_1}_{i',j'}$. Moreover $M$ induces a bijection between the edges of $F^{G_2}_{i,j}$ with $b_i$ as a vertex and the edges of $F^{G_1}_{i',j'}$ with $b_i'$ as a vertex. We know that the edges of $F^{G}_{i,j}$ incident on $b^{G}_i$ are $(b^{G}_i,b^{G}_i+b^{G}_k)$ for $k \notin \{ i,j \}$. With this information, we deduce that $M$ induces a bijection between the vertices of the form $b^{G_2}_k$ where $k \notin \{i',j'\}$ and the vertices of the form $b^{G_1}_k$ where $k \notin \{i,j\}$. Hence, $M$ induces a permutation $\sigma$ between the vertices of $H_{Del_{G_2}}(O)$ that are of the form $b^{G_2}_j$ and the vertices of $H_{Del_{G_1}}(O)$ of the form $b^{G_1}_j$ for $0 \leq j \leq n$.
Hence, we have $M(\sum_{k=0}^{n}b^{G_1}_{k})=M(O)=\sum_{k=0}^{n}b^{G_2}_{k}=O$. 
Hence, $t=O$ and $M$ is an orthogonal transformation and we have $M(b^{G_1}_i)\cdot M(b^{G_1}_j)= b^{G_1}_i \cdot b^{G_1}_j$. Putting it together, we have a permutation $\sigma \in S_{n+1}$ such that  $ b^{G_2}_i \cdot b^{G_2}_j= b^{G_1}_{\sigma(i)}\cdot b^{G_1}_{\sigma(j)}$ for all integers $0 \leq i,j \leq n$. 
This implies that $Q(G_2)Q^{t}(G_2)=Q(G_2)^2=\sigma Q^2(G_1) \sigma^{-1}=(\sigma Q(G_1) \sigma^{-1}) (\sigma Q(G_1)\sigma^{-1})$. But, we know that $Q(G_1)$ and $Q(G_2)$ are positive semidefinite and hence $\sigma Q(G_1) \sigma^{-1}=\sigma Q(G_1) \sigma^{t}$ is also positive semidefinite. By the unique squares lemma \cite{HorJoh85}, we can conclude that $\sigma Q(G_1)\sigma^{-1}$ is the unique positive semidefinite square root of $Q^{2}(G_2)$ and hence $Q(G_2)=\sigma Q(G_1) \sigma ^{-1}$. This shows that $G_1$ and $G_2$ are isomorphic.
\end{proof}


\begin{remark}
There are simpler constructions that also have the property shown in Theorem \ref{geomcomb_theo}. For example for every connected graph $G$ on $n+1$ vertices, associate an $n$-dimensional simplex given by $\mathcal{S}(G)=CH(b_0,\dots,b_{n})$ where $b_0,\dots,b_{n}$ are the rows of the Laplacian of $G$. A argument similar to last part of the proof of Theorem \ref{geomcomb_theo} shows that: Let $G_1$ and $G_2$ be connected graphs $\mathcal{S}(G_1)$ is congruent to $\mathcal{S}(G_2)$ if and only if $G_1$ and $G_2$ are isomorphic. On the other hand, Theorem \ref{geomcomb_theo} is more ``canonical'' in the sense that the polytopes $H_{Del(G)}(O)$ have a geometric interpertation in terms of the Laplacian lattice while the simplex $\mathcal{S}(G)$ does not seem to have any direct interpertation. 
\end{remark}
\begin{remark}
We know that the lengths of the edges of $H_{Del_{G}}(O)$ is essentially the degrees of different vertices of $G$. The volume of $H_{Del_{G}}(O)$ is essentially the number of spanning trees of $H_{Del_{G}}(O)$. Is there such an interpretation for the volumes of the other faces of $H_{Del_{G}}(O)$ in the appropriate measures?
\end{remark}

\section{On the number of graphs with a given Laplacian lattice} \label{Count_Sect}

In Lemma \ref{Laptree_lem} of the previous section, we observed that the Laplacian lattice of any tree is the root lattice $A_n$.
This observation raises the problem of counting the number of graphs that have $A_n$ as their Laplacian lattice. The matrix-tree theorem gives an answer to the problem.
\begin{lemma}\label{numAn_lem}
The number of graphs that have $A_n$ as their Laplacian lattice is exactly $(n+1)^{n-1}$ i.e. the number of labelled trees on $n+1$ vertices.
\end{lemma}
\begin{proof}
By Lemma \ref{Laptree_lem} of the previous section, we know that the root lattice $A_n$ is the Laplacian lattice of any tree on $n+1$ vertices. Conversely,  any connected graph on $n+1$ vertices that is not a tree must contain at least two spanning trees and hence by Lemma \ref{covol_lem}, its Laplacian lattice must be a sublattice of $A_n$ with covolume strictly greater than one.
\end{proof}

Lemma \ref{numAn_lem} raises the following natural problem: 

Given a sublattice $L$ of $A_n$. Count the number of labelled connected graphs whose Laplacian lattice is $L$. 


We denote the number of undirected connected graphs that have $L$ as their Laplacian lattice as $N_{Gr}(L)$.
Note that $N_{Gr}(L)$ is non-zero only if $L$ is the Laplacian lattice of a connected graph. 
Our main result in this section is an upper bound on the number $N_{Gr}(L_G)$ for a Laplacian lattice $L_G$ in terms of the number of different  Delaunay triangulations of $L_G$ under the simplicial distance function $d_{\triangle}$.  As a corollary,
 we show that for a Laplacian lattice $L_G$ of a graph where every pair of vertices are connected by an edge, we have $N_{Gr}(L_G)=1$.





From now on, when we say Delaunay triangulation, we mean that the Delaunay 
triangulation under the simplicial distance function $d_{\triangle}$, 
see Subsection \ref{Deltri_Subsect} for a discussion on Delaunay triangulation 
under polyhedral distance function and Subsection \ref{DeltriLap_Subsect} for 
a discussion on the Delaunay triangulation of Laplacian lattices under 
the simplicial distance function $d_{\triangle}$.

 We now make precise what two triangulations are different means. 

For a triangulation $T$, let $H_{T}(O)$ be the union of simplices in the triangulation that have the origin $O$ as a vertex.
We say $T_1$ and $T_2$ are the same if $H_{T_1}(O)=H_{T_2}(O)$, otherwise they 
are different.

\begin{theorem}\label{uppbound_theo} Let $Del(G)$ be the Delaunay triangulation of $L_G$ defined by the graph $G$ under the distance function $d_{\triangle}$ and $N_{Del}(L_G,\triangle)$ be the number of different Delaunay triangulations of $L_G$, we have $N_{Gr}(L_G) \leq N_{Del}(L_G,\triangle)$.\end{theorem}

\subsection{Proof of Theorem \ref{uppbound_theo}}
  






We know that every graph provides a Delaunay triangulation of its Laplacian lattice (Theorem \ref{LapDel_theo}).
We show that if $H_{Del_{G_1}}(O)=H_{Del_{G_2}}(O)$ then $Q(G_1)=Q(G_2)$. In other words, $G$ can be uniquely recovered from $H_{Del_{G}}(O)$. This would imply that two graphs cannot give rise to the same Delaunay triangulation and Theorem \ref{uppbound_theo}
then follows.

\begin{lemma}
 Let $G_1,~G_2$ be connected graphs, $H_{Del_{G_1}}(O)=H_{Del_{G_2}}(O)$ if and only if $Q(G_1)=Q(G_2)$. 
\end{lemma}

\begin{proof}
Indeed if $Q(G_1)=Q(G_2)$ then $H_{Del_{G_1}}(O)=H_{Del_{G_2}}(O)$. To show the converse, we describe an algorithm to uniquely recover $Q(G_1)$ from $H_{Del_{G_1}}(O)$.

Define the set $C_i$ as follows:

$C_i= \{ p=(p_0,\dots,p_{n}) \in \mathbb{R}^{n+1}|$ $ p_i \geq 0$ and $ p_j \leq 0$ for all $j \neq i \}$.

Indeed, we verify that $C_i$ is a cone. Consider the set $H_{Del_{G_1}}(O)|C_i$ of vertices of $H_{Del_{G_1}}(O)$ that are contained in the cone $C_i$. Pick a vertex $v$, say in $H_{Del_{G_1}}(O)|C_i$ that maximizes the value of the i-th coordinate.

 First, a vertex with this property exists since the vertex $b_i(G_1)$ is contained in $C_i$ since $b_{ij}(G_1)$ the $j$-th coordinate of $b_i(G_1)$ satisfies $b_{ij}(G_1) \leq 0$ for $i \neq j$. 
Furthermore, we claim that $v$ is unique and is equal to $b_i(G_1)$. Now, assume that $v \neq b_i(G_1)$. By Lemma \ref{verchar_lem}, we know that $v=b_{i_0}(G_1)+\dots+b_{i_k}(G_1)$ for some $0 \leq  k \leq n-1$ and $i_j$s all being distinct. Denote $S_k=\{i_0,\dots,i_k\}$. Since, $b_{ij}(G_1) \leq 0$ for all $i \neq j$ we know that $v_i=b_{ii}(G_1)$ and hence, $i \in  S_k$ and $b_{ij}(G_1)=0$ for all $j \in S_k/\{i\}$. We know that $v_j \geq 0$ for all $j \in S_k$.
But since $v \in C_i$, this means that $v_j=0$ for all $j \in S_k/ \{i\}$. This means that $b_{jk}(G_1)=0$ for all $j \in S_k/  \{i\}$ and $k \in (V(G)/S_k) \cup \{ i \}$. Hence, there are at least two components of $G$ that are not connected, namely the induced subgraphs of vertices $S_k/\{i\}$ and $(V(G)/S_k) \cup \{i\}$. This contradicts our assumption that $G$ is connected. This shows that if  $H_{Del_{G_1}}(O)=H_{Del_{G_2}}(O)$ then $b_i(G_1)=b_i(G_2)$ for all $0 \leq i \leq n$ and hence $Q(G_1)=Q(G_2)$. This concludes the proof of the lemma.

\end{proof}

This shows  that each graph $G$ with $L_G=L$ contributes to a different Delaunay triangulation of $L$ and hence, $N_{Gr}(L_G) \leq N_{Del}(L_G,\triangle)$. This concludes the proof of Theorem \ref{uppbound_theo}.

We know from the results in \cite{AmiMan10} that multigraphs where every pair of vertices
are connected by an edge have a unique Delaunay triangulation. As a corollary we obtain:

\begin{corollary}
If $L_G$ is the Laplacian lattice of multigraph $G$ such that every pair of vertices are connected by an edge,
then $N_{Gr}(L_G)=1$. 
\end{corollary}








\section{Covering and Packing problems on the Laplacian lattice}\label{CovPacProb_Sect}

Covering and Packing problems on lattices have been widely studied, see Conway and Sloane \cite{ConSlo99} 
for a general introduction and Zong and Talbot \cite{ZongTal99} or Martinet \cite{Mar03} for a 
more specialised treatment of the subject. 

  Given a lattice $L$ and a convex body $\mathcal{P}$,
we study to how ``optimally'' do $L$-translates of $C$ pack or cover $\mathbb{R}^{n+1}$. 
For a sublattice $L$ of $A_n$, we define the packing density $\gamma_{C}(L)$ and the covering density $\theta_{C}(L)$ as:
 \begin{equation}
    \begin{split}
    \gamma_{C}(L)=\mathrm{Pac}_C(L)/{((n+1)Cvol)}^{1/n}(L)\\
    \theta_{C}(L)=\mathrm{Cov}_C(L)/{((n+1)Cvol)}^{1/n}(L)
     \end{split}
  \end{equation}
where $\mathrm{Pac}_C(L)$ and $\mathrm{Cov}_C(L)$ is the packing and covering radius of $L$ with respect to $C$ and $Cvol(L)$
is the covolume of the lattice with respect to $A_n$. 

Remark that in the standard definition of packing and covering density, the volume of the lattice $L$ appears in the place of $n+1$ times the covolume of $L$. Observe that the two notions are ``equivalent'' up to a factor that depends only on the dimension of $L$ and are interchangeable since we are interested in determining lattices with good packing and covering densities in a given dimension and the order of the quotient group $A_n/L$ is equal to the ratio of the volumes of $A_n$ and $L$ i.e. $Vol(L)=Vol(A_n).|A_n/L|$ (see Lecture V, Theorem 20 of Siegel \cite{Siegel89} for a proof).


The lattice packing and covering problem respectively is to find lattices that pack $\mathbb{R}^{n}$ most densely and cover
 $\mathbb{R}^{n}$ most economically i.e. lattices that maximise $\gamma_{C}(.)$ and minimise $\theta_{C}(.)$.
Explicit constructions of lattices that solve the sphere packing and covering problems are known only for lower dimensions. In two dimensions, the hexagonal lattice produces the densest sphere packing.







We will use the fact that the covering and packing radius of the Laplacian lattice have a combinatorial interpretation in terms of the underlying graph to obtain a formula for the covering and packing density of the Laplacian lattice. We will use this information to show that in the space of Laplacian lattices of undirected connected graphs, the Laplacian lattices of graphs that are highly connected such as Ramanujan graphs have relatively good packing and covering properties.



As a direct consequence of the formulas for the packing and covering radius (Theorem \ref{pacrad_theo} and  Theorem \ref{coverrad_theo}) we obtain:

\begin{corollary} \label{covpac_cor}
The packing radius and covering radius of the Laplacian lattice $L_G$ respectively are:
\begin{gather}
 \gamma_{\triangle}(L_G)=\mathrm{MC}_{1}(G)/((n+1)(\prod_{i=1}^{n} \lambda_i)^{1/n})\\ 
 \theta_{\triangle}(L_G)=(\sum_{i=1}^{n} \lambda_i)/(2(n+1)(\prod_{i=1}^{n} \lambda_i)^{1/n})  
\end{gather}
where $\lambda_1 \leq \lambda_2 \dots \leq \lambda_n$ are the non-zero eigenvalues of the Laplacian matrix of $G$. 
\end{corollary}

This immediately gives a lower bound for the covering density of a Laplacian lattice:

\begin{theorem}{\bf (Lower bounds for the covering density of Laplacian lattices)}\label{covlower_theo}
The covering density of the Laplacian lattice is at least $n/{2(n+1)}$.
\end{theorem}

\begin{proof}
By Corollary \ref{covpac_cor}, we know that $\theta_{\triangle}(L_G)=(\sum_{i=1}^{n} \lambda_i)/(2(n+1)(\prod_{i=1}^{n} \lambda_i)^{1/n})$ 
where $\lambda_1 \leq \lambda_2 \dots \leq \lambda_n$ are the non-zero eigenvalues of the Laplacian matrix.
Now we use the fact that the Laplacian matrix is a positive semidefinite matrix along with the AM-GM inequality to obtain
that $\theta_{\triangle}(L_G) \geq n/{2(n+1)}$.
\end{proof}


We consider the problem of minimising the covering density and maximising the packing density of the Laplacian lattice over all connected graphs with a given number of vertices. First we consider the covering density case.
Suppose that $\lambda_1,\dots,\lambda_n$ are arbitrary positive real numbers then the quantity $\mathrm{Cov}_{\triangle}(L_G)$ is minimised  if $\lambda_1=\lambda_2=\cdots=\lambda_n$. But then in our case, these numbers are eigenvalues of the Laplacian matrix. Nevertheless, we would like the eigenvalues of the Laplacian matrix to be ``clustered''. This suggests that graphs with good expansion properties would be suitable. To make this intuition precise, we need the notion of a Ramanujan graph, see the survey of Horory et al. \cite{HorLinWig06} for a more detailed discussion on the topic.

\begin{definition}
 A $d$-regular graph is called a Ramanujan graph if $\lambda^{A}(G) \leq 2\sqrt{d-1}$, where $\lambda^{A}(G)=max\{|\lambda^{A}_2|,|\,\lambda^{A}_{n+1}|\}$ and $d=\lambda^{A}_1 \geq \dots \geq\lambda^{A}_{n+1}$ are the eigenvalues of the adjacency matrix of $G$.  
\end{definition}

Using the fact that $\lambda_i=d-\lambda^{A}_i$ we have:
\begin{lemma}
The non-zero eigenvalues of the Laplacian matrix of a Ramanujan graph are located in the interval $[d-2\sqrt{d-1}, d+2\sqrt{d-1}]$.
\end{lemma}

Suppose that the graph is a Ramanujan graph, then we know that the eigenvalues of its Laplacian matrix are concentrated around $d$ (the degree of the graph) in an interval of width $2\sqrt{d-1}$. More precisely, we have:
$d-2\sqrt{d-1} \leq \lambda_i \leq d+2\sqrt{d-1}$ for every non-zero eigenvalue of the Laplacian. 
Using this information, we will obtain an upper bound on the covering density of a Ramanujan graph. 

\begin{lemma}\label{ramspa_lem} Let $G$ be a $d$-regular Ramanujan graph, we have $d+2\sqrt{d-1} \geq (\prod_{i=1}^{n}\lambda_i)^{1/n} \geq d-2\sqrt{d-1}$ \end{lemma}

\begin{proof}
Since $G$ is a $d$-regular Ramanujan graph we have $d-2\sqrt{d-1} \leq \lambda_i \leq d+2\sqrt{d-1}$ for every $i$ from $1 \dots n$. Using this information, we have $(d+2\sqrt{d-1}) \geq (\prod_{i=1}^{n} \lambda_i)^{1/n} \geq (d-2\sqrt{d-1})$.
\end{proof}

\begin{theorem}{\bf (Covering Density of Ramanujan graphs)}\label{ramcov_theo}
Let $G$ be a $d$-regular Ramanujan graph then $\theta_{\triangle}(L_G) \leq (\frac{d}{4(d-2 \sqrt{d-1})})$.
\end{theorem}

\begin{proof}
We have $\sum_{i=1}^{n} \lambda_i={(n+1)\cdot d}/2$ and by Lemma \ref{ramspa_lem}
we have $\prod_{i=1}^{n}\lambda_i  \geq d-2\sqrt{d-1}$. By Corollary \ref{covpac_cor}, we know that $\theta_{\triangle}(L_G)=(\sum_{i=1}^{n} \lambda_i)/(2(n+1)(\prod_{i=1}^{n} \lambda_i)^{1/n})$ 
where $\lambda_1 \leq \lambda_2 \dots \leq \lambda_n$ are the non-zero eigenvalues of the Laplacian matrix. 
Hence, we obtain $\theta_{\triangle}(L_G) \leq (\frac{d}{4(d-2 \sqrt{d-1}}))$.
\end{proof}

\begin{remark}
Note that we crucially use the fact that the eigenvalues of a Ramanujan graph are 
concentrated in a small interval. It is not clear if we can obtain upper bounds on the covering density for a general graph. 
\end{remark}

We do not know if the converse of Theorem \ref{ramcov_theo} also holds. More precisely,

\begin{question}
Suppose that the covering density of the Laplacian lattice of a graph is upper bounded by a suitably chosen constant $c$, then is it true that the graph is Ramanujan?
\end{question}

We now consider the problem of maximising the packing density of the Laplacian lattice. The formula for packing density presented in Theorem \ref{covpac_cor} suggests that graphs that maximise the packing density have high minimum cut and a relatively small number of spanning trees. We will now provide a lower bound on the packing density of Ramanujan graphs.

\begin{lemma}
Let $G$ be a $d$-regular Ramanujan graph, then $\gamma_{\triangle}(G) \geq \frac{(d-2\sqrt{d-1})}{2(n+1)(d+2\sqrt{d-1})}$.
\end{lemma}
\begin{proof}
By Corollary \ref{covpac_cor} we have: $\gamma_{\triangle}(G)=\frac{\mathrm{MC}_{1}(G)}{(n+1)(\prod_{i=1}^{n}\lambda_i)^{1/n}}$.
By Lemma \ref{ramspa_lem} we have: $\prod_{i=1}^{n}\lambda_i  \leq d+2\sqrt{d-1}$. 
We now obtain a lower bound on $\mathrm{MC}_{1}(G)$. We observe that 
the size of the cut $S$ can be written as $u_S\cdot Q(G)\cdot u_S^{t}/2$
where is the indicator vector of $S$ i.e.

\begin{math}
{u_S}_i=
\begin{cases}
 \hspace{0.2cm} 1 \text{, if the vertex with index $i$ is in $S$}, \\
 -1 \text{, otherwise.}
\end{cases}
\end{math}

We now observe that: 
\begin{gather}\notag
\min_{S \notin \{V, \emptyset\}}\frac{ u_S \cdot  Q(G) \cdot u_S^{t}}{2} \geq \min_{u \in \mathbb{S}^{1}} \frac{u \cdot Q(G) \cdot u^{t}/2}{2} \cdot \min_{S \notin \{V, \emptyset\}}{u_S \cdot u_S} \geq \min_{u \in \mathbb{S}^{1}} \frac{u \cdot Q(G) \cdot u^{t}}{2}=\lambda_n/2
\end{gather}
Now since the graph is a Ramanujan graph we know that $\lambda_n \geq d-2\sqrt{d-1}$. Hence, $\gamma_{\triangle}(G) \geq \frac{d-2\sqrt{d-1}}{2(n+1)(d+2\sqrt{d-1})}$.
\end{proof}

A natural question is that whether the lower bound on the packing density that we obtained for the Laplacian lattice of a  Ramanujan graph is the best possible. Note that a trivial upper bound
for $\gamma_{\triangle}(G)$ is $d$ for a $d$-regular graph $G$. It would be interesting to study the converse of this question, namely, suppose that the Laplacian lattice of a $d$-regular graph has a high packing density then does the graph have ``high'' connectivity?






%



\section{Appendix}

\subsection{Sequences of Lattices and their limit} \label{Laqseq_subsect}

 A sequence of lattices $\{L_n\}$ is said to {\bf converge} to a lattice $L_{\ell}$ if for every $\delta>0$ and $q \in L_{\ell}$ there exists a positive integer $N(\delta,q)$ and a family of bijective maps $\phi_{N}: L_{\ell} \rightarrow L_{N}$ for $N \geq N(\delta,q)$ such that $||q-\phi_{N}(q)||_2< \delta$ for all $N \geq N(\delta,q)$. Similarly, a sequence of lattices $\{L_n\}$ is said to {\bf uniformly converge} to a limit $L_{\ell}$ if for every $\delta>0$ there exists an integer $N(\delta)$ and a family of bijective maps $\phi_{N}: L_{\ell} \rightarrow L_{N}$ for $N \geq N(\delta)$ such that $||q-\phi_{N}(q)||_2< \delta$ for all $q \in L_{\ell}$ and for all $N \geq N(\delta)$. Note that since convex polytopes are topologically equivalent to Euclidean balls, the notion of convergence of a sequence of lattices with respect to polyhedral distance functions is equivalent to the above notion.

Typically, we take a basis $\mathcal{B}$ of the lattice
and add an infinitesimal perturbation $\mathcal{B}^{\epsilon}$ 
to it and consider the lattice generated by the perturbed lattice. 
As $\epsilon$ tends to zero, the sequence of lattices converges to the lattice generated by $\mathcal{B}$, 
but not necessarily uniformly.
Observe that there is a natural bijection $\phi_{\epsilon}$ from $L_{\ell}$ to $L_{\epsilon}$ 
induced by the basis $\mathcal{B}$ defined as $\phi_{\epsilon}(\mathcal{B}\alpha)=\mathcal{B}^{\epsilon}\alpha$. 
In the following lemma we prove that the shortest vector and packing radius are preserved under the limit, see the book of Gruber and Lekkerkerker \cite{GruLek87} for a more detailed treatment of lattices under perturbation.

\begin{lemma}
Consider a sequence of lattices $\{L_{n}\}$ that converge
to a lattice $L_{\ell}$. For any convex polytope $\mathcal{P}$, the length of the shortest vector and packing radius under the distance  function $d_{\mathcal{P}}$ is preserved under limit. More precisely we have:
\begin{enumerate} 
 \label{short_claim}\item $\lim_{n \rightarrow \infty}\nu_{\mathcal{P}}(L_n)=\nu_{\mathcal{P}}(L_{\ell})$ 
 \label{pac_claim} \item $\lim_{n \rightarrow \infty}\mathrm{Pac}_{\mathcal{P}}(L_n)=\mathrm{Pac}_{\mathcal{P}}(L_{\ell})$.
\end{enumerate}
\end{lemma}

\begin{proof}

i. Given an $\epsilon>0$, we show that there exists a positive integer $n(\epsilon)$ such that $|\nu_{\mathcal{P}}(L_{\ell})-\nu_{\mathcal{P}}(L_n)|\leq \epsilon$
for all $n \geq n(\epsilon)$. Consider the set of shortest vectors $M_1$ of the lattice $L_{\ell}$.
Let $d$ be the the difference between the norm of the shortest vector and second shortest vector in the distance function $d_{\mathcal{P}}$. Let $\delta$ be the minimum of $d/2$ and $\epsilon/2$. Take a ball of radius $R$ much larger than $\nu_{\mathcal{P}}(L_{\ell})$ centered at the origin, in fact taking $R$ to be $2 \nu_{\mathcal{P}}(L_{\ell})+d$ suffices. Using our assumption that $\{L_n\}$ converges to $L_{\ell}$ and the fact that number of points in $L_{\ell}$ that are contained inside the ball of radius $R$ centered at the origin is finite, we know that there exists a positive integer $n(\delta)$ and a bijection $\phi_n: L_{\ell} \rightarrow L_{n}$ such that for $n \geq n(\delta)$ we have $d_{\mathcal{P}}(q,\phi_n(q))<\delta$ for points $q$ in $L_{\ell}$ that are contained in the ball of radius $R$. For a point $p$, we call $d_Q(O,p)$ the norm of the point. 
Now we claim that for all $n \geq n(\delta)$ any shortest vector of $L_n$ belongs to $\mathcal{P}(q',\delta)$ where $q' \in M_1$. 
The argument is as follows: any element in $L_n$ for $n \geq n(\delta)$ that does not belong to $\mathcal{P}(q',\delta)$ for $q' \in M_1$ must have norm strictly greater than $\nu_{\mathcal{P}}(L_{\ell})+d/2$. On the other hand, we know that there exists an element of $L_n$ that is contained in $\mathcal{P}(q',\delta)$ for some $q' \in M_1$ and hence has norm at most $\nu_{\mathcal{P}}(L_{\ell})+\delta$. Since $\delta \leq d/2$, we arrive at a contradiction. We finally note that every point  contained in the ball $\mathcal{P}(q',\delta)$ for $q' \in M_1$ has norm between $\nu_{\mathcal{P}}(L_\ell)+\delta$ and $\nu_{\mathcal{P}}(L_{\ell})-\delta$. 
Hence, $|\nu_{\mathcal{P}}(L_n)-\nu_{\mathcal{P}}(L_{\ell}| \leq 2\delta \leq \epsilon$ for all $n \geq n(\delta)$. 

ii. For the packing radius, the argument is essentially the same as the argument for shortest vectors except that in this case we carry out the argument on the $\mathcal{P}$-midpoints of the lattice points with the origin rather than on the lattice points.
We consider the $\mathcal{P}$-midpoints of every point of $L_{\ell}$ with the origin. Let $M_1$ be the set of $\mathcal{P}$-midpoints of every point in $L_{\ell}$ that define its packing radius. Let $d$ be the difference between the packing radius and the norm of $\mathcal{P}$-midpoints that are closer to the origin expect to the $\mathcal{P}$-midpoints that define the packing radius. Take $\delta$ to be the minimum of $d/2$ and $\epsilon/2$. Take a ball of radius $R$ much larger than $\mathrm{Pac}_{\mathcal{P}}(L_{\ell})$ centered at the origin. Since $\{L_n\}$ converges to $L_{\ell}$ and since the $\mathcal{P}$-midpoints vary continuously with perturbation, we know that there exists a integer $n(\delta)$ and a bijection $\phi_n:L_{\ell} \rightarrow L_{n}$ such that  $d_{\mathcal{P}}(b,\phi_N(b))<\delta$ for $\mathcal{P}$-midpoints in the ball of radius $R$.
Now we claim that for $n \geq n(\delta)$ the $\mathcal{P}$-midpoint that defines the packing radius of $L_n$ must be contained in $\mathcal{P}(b,\delta)$ for some $b \in M_1$. This follows from the fact that for $n \geq n(\delta)$ any element in $L_n$ that does not belong to $\mathcal{P}(b',\delta)$ where $b' \in M_1$ must have norm strictly greater than $\mathrm{Pac}_{\mathcal{P}}(L_{\ell})+d/2$. On the other hand, we know that the bisectors of  elements of $L_n$ that are contained in $\mathcal{P}(q',\delta)$ have norm at most $\mathrm{Pac}_{\mathcal{P}}(L_{\ell})+\delta$.
 Since $\delta \leq d/2$, we arrive at a contradiction. We finally note that every element of $\mathcal{P}(q',\delta)$ for $q' \in M_1$  contained in the ball has norm between $\mathrm{Pac}_{\mathcal{P}}(L_\ell)+\delta$ and $\mathrm{Pac}_{\mathcal{P}}(L_{\ell})-\delta$. Hence, $|\mathrm{Pac}_{\mathcal{P}}(L_n)-\mathrm{Pac}_{\mathcal{P}}(L_{\ell})| \leq 2\delta \leq \epsilon$ for all $n \geq n(\delta)$. 

\end{proof}

{\bf Acknowledgements:} The author wishes to thank Girish Varma for the various clarifying discussions and Omid Amini 
for directing him to the Riemann-Roch theory for graphs and for the collaboration on the previous work in this topic.
As always, he is heavily indebted to his parents for being a constant source of support and encouragement.






}


\begin{thebibliography}{0}

\bibitem{Art06} I. V. Artamkin, \textsl{Discrete Torelli Theorem}, Matematicheski˘ı Sbornik ({\bf  197:8}), 3–16 (2006).

 \bibitem{AmiMan10} O. Amini and M. Manjunath, \textsl{Riemann-Roch for Sub-lattices of the Root Lattice $A_n$}, Electronic Journal of Combinatorics ({\bf 17(1)}), (2010).

\bibitem{BacHarNag97} R. Bacher, P. Harpe and T. Nagnibeda, \textsl{The Lattice Of Integral Flows And The Lattice Of Integral Cuts On A Finite Graph}, Bull. Soc. Math. France, {\bf 125} (1997), 167--198.

\bibitem{BaNo07} M. Baker and S. Norine, \textsl{Riemann-Roch and Abel-Jacobi Theory on a Finite Graph}, Advances in Mathematics ({\bf 215 (2)}) (2007), 766--788.


\bibitem{BeChKrOv08} M. Berg, Otfried Cheong, Marc van Kreveld and Mark Overmars, \textsl{Computational geometry: algorithms and applications}, Springer: 3rd edition, (2008), 191--196.

\bibitem{Big97} N. Biggs, \textsl{Chip Firing and the Critical Group of a Graph}, \textsl{Journal of Algebraic Combinatorics}, {\bf 9},(1999), 25--45.

\bibitem{Bourbaki} N. Bourbaki, \textsl{Lie Groups and Lie Algebras}, Springer, 1989.

\bibitem{ConSlo99} J.H. Conway and N.J.A Sloane, \textsl{Sphere Packings, Lattice and Groups}, Springer, (1999) Third edition.

\bibitem{Che91} J. Cheeger, \textsl{Critical points of distance functions and applications to geometry}, Geometric Topology: recent
developments,Montecatini Terme, Springer Lecture Notes, {\bf 1504} (1991), 1--38.


\bibitem{ErdRys94} R. M. Erdahl and S.S. Rshykov, \textsl{On lattice dicings}, European Journal of Combinatorics, {\bf 15}, 459--481, 1994.

\bibitem{Gauss66} C. F. Gauss, \textsl{Disquisitiones arithmeticae}, Springer-Verlag, 1966. 

\bibitem{GieJoh08} J. Giesen and M. John, \textsl{The Flow Complex: A Data Structure For Geometric Modeling}, Computational Geometry: Theory and Applications ({\bf 39(3)}) (2008), 178--190.

\bibitem{GruLek87} P. M. Gruber, C. G. Lekkerkerker, \textsl{Geometry of numbers}, North-Holland Mathematical Library, Elsevier, Second Edition (1987). 

\bibitem{GodRoy01} C. Godsil and G. Royle. \textsl{Algebraic Graph Theory}, Springer-Verlag, (2001).


\bibitem{HorJoh85} R. A. Horn and C. R. Johnson, \textsl{Matrix Analysis}, Cambridge University Press, (1985), 405.

\bibitem{HorLinWig06} S. Horry, N. Linial, and A. Wigderson, \textsl{Expander Graphs and their applications}, Bulletin of the American Mathematical Society, {\bf 43(4)} (2006), pages 439--561.

\bibitem{KotSun98} M. Kotani and T. Sunada, \textsl{Advances in Applied Mathematics}, {\bf 24}, 2000, 89--110.

\bibitem{Lor08} D. Lorenzini, Smith normal form and Laplacians, \textsl{Journal of Combinatorial Theory}, Series B, {\bf 98 (6)}, 2008, 1271--1300.



\bibitem{Min55} H. Minkowski, \textsl{Geometrie der Zahlen}, Chelsea, reprint, 1953.

\bibitem{Mar03} J. Martinet, \textsl{Perfect Lattices in Euclidean Spaces}, Springer, 2003. 

\bibitem{PeeStu98} I. Peeva and B. Sturmfels, \textsl{Generic Lattice Ideals}, Journal of the American Mathematical Society, 11{\bf(2)}, 363--373, 1998.


\bibitem{Sca08} H. Scarf, \textsl{The structure of the complex of maximal lattice point free bodies for a matrix of size $n+1 \times n$}, 
               Bolyai Society Mathematical Studies, {\bf 19}, 1--32,  2008.

\bibitem{Schneider93} R. Schneider, Convex Bodies: The Brunn-Minkowski Theory,  Encyclopedia of Mathematics and its applications 41, Cambridge University Press, (1993). 


\bibitem{Schrijver86} A. Schrijver, \textsl{Theory of linear and integer programming}, Wiley-Interscience, (1986). 

\bibitem{Sho09} F. Shokrieh, Chip-Firing Games, \textsl{G-Parking Functions, and an Efficient Bijective Proof of the Matrix-Tree Theorem}, available at arXiv:0907.4761v1 [math.CO].

\bibitem{SikSchVal08} M. D. Sikiric, A. Sch\"urmann, F. Vallentin \textsl{Complexity and algorithms for computing Voronoi cells of lattices}, Mathematics of Computation, 78 {\bf(267)}, 1713-1731, 2009


\bibitem{Siegel89} C. L. Siegel, \textsl{Lectures on the Geometry of Numbers}, Springer (1989).

\bibitem{Tardos88} G. Tardos, \textsl{Polynomial bound for a chip firing game on graphs},
SIAM Journal on Discrete Mathematics, {\bf 1 (3)}, (1988), 397--398.


\bibitem{CapViv10} L. Caporaso and F. Viviani, \textsl{Torelli theorem for graphs and tropical curves}, Duke Mathematical Journal, Duke Mathematical Journal, {\bf 153(1)}, (2010), 129-171. 

\bibitem{ZongTal99} C. Zong, J. Talbot, \textsl{Sphere Packings}, Springer, (1999).

\end{thebibliography}
\end{document}